\documentclass{amsart}

\usepackage[T1]{fontenc}
\usepackage{enumerate, amsmath, amsfonts, amssymb, amsthm, mathrsfs, wasysym, graphics, graphicx, xcolor, url, hyperref, hypcap, shuffle, xargs, multicol, overpic, pdflscape, multirow, hvfloat, minibox, accents, array, multido, xifthen, a4wide, ae, aecompl, blkarray, pifont, mathtools, etoolbox, cancel}
\usepackage[linesnumbered,ruled,lined]{algorithm2e}
\definecolor{darkblue}{rgb}{0.0,0.0,0.7}
\hypersetup{colorlinks=true, breaklinks, citecolor=darkblue, linkcolor=darkblue}
\usepackage[all]{xy}
\usepackage[bottom]{footmisc}
\usepackage{tikz}
\usetikzlibrary{trees, decorations, decorations.markings, shapes, arrows, matrix, calc, fit, intersections, patterns, angles, automata, positioning}
\usepackage[noabbrev,capitalise]{cleveref}

\newtheorem{theorem}{Theorem}
\newtheorem{proposition}[theorem]{Proposition}
\newtheorem{lemma}[theorem]{Lemma}
\newtheorem{corollary}[theorem]{Corollary}

\newtheorem{definition}[theorem]{Definition}
\newtheorem{example}[theorem]{Example}
\newtheorem{remark}[theorem]{Remark}

\newcommand{\fS}{\mathfrak{S}} 

\newcommand{\set}[2]{\left\{ #1 \;\middle|\; #2 \right\}} 
\newcommand{\ssm}{\smallsetminus} 
\newcommand{\symdif}{\,\triangle\,} 
\newcommand{\eqdef}{\mbox{\,\raisebox{0.2ex}{\scriptsize\ensuremath{\mathrm:}}\ensuremath{=}\,}} 
\newcommand{\defn}[1]{\textsl{\color{darkblue} #1}} 
\newcommand{\ie}{\textit{i.e.}~} 
\newcommand{\apriori}{\textit{a priori}} 
\DeclareMathOperator{\inv}{inv} 
\DeclareMathOperator{\ninv}{ninv} 
\DeclareMathOperator{\moveU}{moveU} 
\DeclareMathOperator{\moveD}{moveD}

\newcommand{\automatonA}{\mathbb{A}} 
\newcommand{\automatonP}{\mathbb{P}} 
\newcommand{\automatonU}{\mathbb{U}} 
\newcommand{\automatonD}{\mathbb{D}} 
\newcommand{\length}{\ell} 
\newcommand{\lexmin}{\mathcal{R}} 

\usepackage{todonotes}

\title{Permutree sorting}

\thanks{VP was supported by the French ANR (grants CAPPS~17\,CE40\,0018 and CHARMS~19\,CE40\,0017)}

\author{Vincent Pilaud}
\address[VPi]{CNRS \& LIX, \'Ecole Polytechnique, Palaiseau}
\email{vincent.pilaud@lix.polytechnique.fr}
\urladdr{\url{http://www.lix.polytechnique.fr/~pilaud/}}

\author{Viviane Pons}
\address[VPo]{Universit\'e Paris-Saclay, CNRS, Laboratoire de recherche en informatique, Orsay, France.}
\email{viviane.pons@lri.fr}
\urladdr{\url{https://www.lri.fr/~pons/}}

\author{Daniel Tamayo Jim\'enez}
\address[DTJ]{Universit\'e Paris-Saclay, CNRS, Laboratoire de recherche en informatique, Orsay, France.}
\email{\url{daniel.tamayo-jimenez@lri.fr}}

\begin{document}

\maketitle

\begin{abstract}
Generalizing stack sorting and $c$-sorting for permutations, we define the permutree sorting algorithm.
Given two disjoint subsets~$U$ and~$D$ of~$\{2, \dots, n-1\}$, the $(U,D)$-permutree sorting tries to sort the permutation~$\pi \in \fS_n$ and fails if and only if there are~$1 \le i < j < k \le n$ such that~$\pi$ contains the subword~$jki$ if~$j \in U$ and~$kij$ if~$j \in D$.
This algorithm is seen as a way to explore an automaton which either rejects all reduced words of~$\pi$, or accepts those reduced words for~$\pi$ whose prefixes are all $(U,D)$-permutree sortable.
\end{abstract}


\section{Introduction}

The weak order is a classical lattice on the symmetric group where permutations are ordered by inclusion of inversion sets.
Generalizing the classical Tamari lattice~\cite{Tamari}, N.~Reading studied the lattice quotients of the weak order~\cite{Reading-latticeCongruences}, in particular the Cambrian lattices~\cite{Reading-cambrianLattices,Reading-CoxeterSortable}.
Generalizing and interpolating between the weak order, the Cambrian lattices and the boolean lattice, we introduced permutree lattices in~\cite{PilaudPons-permutrees} based on the combinatorics of certain trees called permutrees.
These permutree lattices have strong combinatorial, geometric and algebraic properties distinguishing them among all lattice quotients of the weak order: for instance, they are the only lattice quotients of the weak order which can be realized as removahedra~\cite{AlbertinPilaudRitter-Removahedra} (\ie polytopes obtained by removing inequalities in the facet description of the classical permutahedra, like the classical associahedra of~\cite{Loday,HohlwegLange}), and they define combinatorial Hopf algebras~\cite{PilaudPons-permutrees} analoguous to the Hopf algebras of C.~Malvenuto and C.~Reutenauer on permutations~\cite{MalvenutoReutenauer} and of J.-L.~Loday and M.~Ronco on binary trees~\cite{LodayRonco}.

Unlike Cambrian lattices which are well-understood for arbitrary finite Coxeter groups~\cite{Reading-cambrianLattices,Reading-CoxeterSortable} with their connections to finite type cluster algebras~\cite{FominZelevinsky-ClusterAlgebrasI,FominZelevinsky-ClusterAlgebrasII}, permutree lattices lack a combinatorial description beyond the symmetric group.
To tackle this problem, this paper still focusses on type~$A$ permutree lattices but from the lens of reduced words for permutations (\ie products of simple transpositions of minimal length).
This approach not covered in \cite{PilaudPons-permutrees} is more suitable to generalizations to finite Coxeter groups.

The goal of this paper is to identify reduced words which correspond to the minimal permutations of their permutree congruence classes.
These permutations are generalizations of the \defn{stack-sortable} permutations introduced by D.~Knuth in his textbook~\cite[Sect.~2.2.1]{Knuth-TAOCP1}, which are characterized by the following equivalent conditions for a permutation~$\pi \in \fS_n$:
\begin{enumerate}[(i)]
\item \label{condition:stackSorting}
$\pi$ is sent to the identity by the stack sorting operator~$S$ defined inductively by~$S(\tau n \rho) \eqdef S(\tau) S(\rho) n$.
\item \label{condition:patternStack}
$\pi$ avoids the pattern $231$ (\ie there is no~$p < q < r$ such that~$\pi_r < \pi_p < \pi_q$).
\item \label{condition:minimalLinearExtentionBinaryTree}
$\pi$ is minimal among all linear extensions of a binary tree on $n$ nodes (seen as a poset, where the nodes are labeled in inorder and the edges are oriented towards the leaves).
\item \label{condition:alignedStack}
For $i < j < k$, the inversion set~$\inv(\pi) \eqdef \set{(\pi_p, \pi_q)}{p < q \text{ and } \pi_p > \pi_q}$ of~$\pi$ contains the inversion $(k,j)$ as soon as it contains the inversion~$(k,i)$.
\item \label{condition:sortableStack}
$\pi$ admits a reduced word of the form~$\pi = c_{I_1} \cdots c_{I_p}$ with nested subsets~$I_1 \supseteq \dots \supseteq I_p$, where~$c_{\{i_1 < \dots < i_j\}} \eqdef s_{i_j} \cdots s_{i_1}$ is a product of the simple transpositions~$s_i \eqdef (i \;\; i+1)$.
\end{enumerate}
It follows from~\eqref{condition:minimalLinearExtentionBinaryTree} that these permutations are counted by the Catalan number~$C_n \eqdef \frac{1}{n+1} \binom{2n}{n}$.

In~\cite{Reading-latticeCongruences, Reading-cambrianLattices, Reading-CoxeterSortable}, N.~Reading defined natural counterparts to conditions~\eqref{condition:minimalLinearExtentionBinaryTree}, \eqref{condition:alignedStack}, and~\eqref{condition:sortableStack} above, parametrized by the choice of a Coxeter element~$c$ in a finite Coxeter group~$W$: the minimality in $c$-Cambrian classes, the $c$-alignment, and the $c$-sortability.
(We skip the general definitions of these conditions here as we stick with the combinatorics of the symmetric group.)
In the situation of the symmetric group~$\fS_n$, we can think of a Coxeter element on~$\fS_n$ as an orientation of an $(n-1)$-path, or equivalently as a partition of~$\{2, \dots, n-1\}$ into two subsets~$U$ and~$D$.
The Cambrian analogues of the conditions~\eqref{condition:patternStack}, \eqref{condition:minimalLinearExtentionBinaryTree}, \eqref{condition:alignedStack} and~\eqref{condition:sortableStack} above are the following equivalent conditions for a permutation~$\pi \in \fS_n$:
\begin{enumerate}[(i')]
\addtocounter{enumi}{1}
\item \label{condition:patternCambrian}
For~$i < j < k$, the permutation~$\pi$ does not contain the subword $jki$ if~$j \in U$ and~$kij$ if~$j \in D$.
\item \label{condition:minimalLinearExtentionCambrian}
$\pi$ is minimal among all linear extensions of a $c$-Cambrian tree on $n$ nodes. A $c$-Cambrian tree is an oriented tree on~$[n]$ where node~$j$ has one parent if~$j \notin U$ and two parents if~$j \in U$, and one child if~$j \notin D$ and two children if~$j \in D$, with an additional local condition at each node similar to the binary search tree condition~\cite{ChatelPilaud}.
\item \label{condition:alignedCambrian}
For~$i < j < k$, if~$\inv(\pi)$ contains~$(k,i)$, then it also contains $(k,j)$ if~$j \in U$ and $(j,i)$ if $j \in D$.
\item \label{condition:sortableCambrian}
$\pi$ admits a reduced word of the form~$\pi = c_{I_1} \cdots c_{I_p}$ with nested subsets~$I_1 \supseteq \dots \supseteq I_p$, where~$c_I \eqdef c_{i_1} \cdots c_{i_{|I|}}$ denotes the subword of~$c \eqdef c_1 \cdots c_{n-1}$ indexed by~$I \eqdef \{i_1 < \dots < i_j\}$.
\end{enumerate}
It turns out that for any Coxeter element~$c$, the permutations satisfying these conditions are still counted by the Catalan number~$C_n$.

The generalization to permutrees consists of taking two subsets~$U$ and~$D$ of~$\{2, \dots, n-1\}$ that are not anymore required to form a partition of~$\{2, \dots, n-1\}$ (they may intersect and may not cover all the set).
It was proved in~\cite{PilaudPons-permutrees, ChatelPilaudPons} that the conditions~(\ref{condition:patternCambrian}'), (\ref{condition:minimalLinearExtentionCambrian}'), and~(\ref{condition:alignedCambrian}') are still equivalent for a permutation~$\pi \in \fS_n$.
We call a permutation \defn{$(U,D)$-permutree minimal} when it satisfies these conditions, \ie when it is minimal (minimal number of inversions) in its $(U,D)$-permutree class.
The number of $(U,D)$-permutree minimal permutations is called $(U,D)$-factorial-Catalan number and admits recursive formulae interpolating between the formulae for the factorial and for the Catalan number.

The objective of this paper is to discuss characterizations of permutree minimal permutations in terms of their reduced words.
In other words, we aim at a condition playing the role of condition~(\ref{condition:sortableCambrian}') and equivalent to conditions~(\ref{condition:patternCambrian}'), (\ref{condition:minimalLinearExtentionCambrian}'), and~(\ref{condition:alignedCambrian}') for arbitrary subsets~$U$ and~$D$ of~$\{2, \dots, n-1\}$.
We believe that understanding this characterization in type~$A$ could lead to a general characterization for all finite Coxeter groups since the condition~(\ref{condition:sortableCambrian}') generalizes to all types for Cambrian lattices.

We first focus on the case where~$U = \varnothing$ and~$D = \{j\}$ for some~${j \in \{2, \dots, n-1\}}$, or the opposite.
To characterize the permutree minimal permutations in terms of their reduced words in that situation, we use two automata~$\automatonU(j)$ and $\automatonD(j)$ defined inductively as shown in \cref{fig:automataRecursive}.
The induction stops at~$\automatonU(n)$ and~$\automatonD(1)$, which are defined by deleting the transitions~$s_n$ and~$s_0$ respectively in \cref{fig:automataRecursive}.
\cref{fig:automataComplete} presents the complete automaton $\automatonU(j)$ after all recursion is done, and \cref{fig:TreePartialOrientations} shows the automata~$\automatonU(2)$, $\automatonD(2)$, $\automatonU(3)$, and~$\automatonD(3)$.
In all these pictures the initial state is marked with ``start'', the accepting states are doubly circled, all transitions are labeled with simple transpositions~$s_i$ for~$i \in [n-1]$, and all missing transitions are loops (we assume the reader familiar with basic automata theory, see for instance~\cite{HopcroftUllman}).
Our main tool is the following statement, proved in \cref{sec:proofPatternAvoidance}.

\begin{theorem}\label{thm:patternAvoidance}
Fix $j \in \{2, \dots, n-1\}$.
The following conditions are equivalent for~$\pi \in \fS_n$:
\begin{itemize}
	\item $\pi$ admits a reduced word accepted by the automaton~$\automatonU(j)$ (resp.~$\automatonD(j)$),
	\item $\pi$ contains no subword $jki$ (resp.~$kij$) with~$i < j < k$.
\end{itemize}
\end{theorem}

Let us warn the reader on the fact that~$j$ is fixed in \cref{thm:patternAvoidance}, while $i$ and~$k$ are arbitrary such that~$1 \le i < j < k \le n$.
A priori, we should try all possible reduced words of~$\pi$ to decide if one is accepted by the automaton~$\automatonU(j)$ (resp.~$\automatonD(j)$).
However, we can show that if~$\pi$ contains no subword $jki$ (resp.~$kij$) with~$i < j < k$ and has a descent~$s_\ell$ distinct from~$s_{j-1}$ (resp.~$s_j$), then it has a reduced word starting with~$s_\ell$ and accepted by the automaton~$\automatonU(j)$ (resp.~$\automatonD(j)$).
In other words, there is no loss of generality in starting constructing a reduced word for~$\pi$ as long as we stay in the states of the top row of~$\automatonU(j)$ (resp.~$\automatonD(j)$).
This yields a simple algorithm to construct a reduced word accepted by~$\automatonU(j)$ (resp.~$\automatonD(j)$).
It also yields natural tree structures on the permutations characterized by \cref{thm:patternAvoidance}, which can be glanced upon in \cref{fig:TreePartialOrientations}.
These algorithmic and combinatorial consequences of \cref{thm:patternAvoidance} are explored in \cref{sec:algorithmicCombinatorialConsequences}.
Most results of \cref{sec:proofPatternAvoidance,sec:algorithmicCombinatorialConsequences} are stated with respect to both automata~$\automatonU(j)$ and $\automatonD(j)$ but proved only for $\automatonU(j)$ as all proofs for~$\automatonD(j)$ are symmetric.

\begin{figure}
	\centerline{
		$\automatonU(j) \eqdef$
		\begin{tikzpicture}[shorten >=1pt, node distance=2cm, on grid, auto, baseline=-1.5cm]
			\node[state,initial,accepting,minimum size=0.5cm] (q_0) {}; 
			\node[state,accepting,minimum size=0.5cm] (q_1) [below= 1.5cm of q_0] {}; 
			\node[state,minimum size=0.5cm] (q_2) [below= 1.5cm of q_1] {}; 
			\node (q_3) [right= 2.5cm of q_0] {$\automatonU(j+1)$};
			\path[->] 
				(q_0) edge node [swap] {$s_{j-1}$} (q_1)
					  edge node {$s_j$} (q_3)
				(q_1) edge node [swap] {$s_j$} (q_2);
		\end{tikzpicture}
		\qquad\qquad
		$\automatonD(j) \eqdef$
		\begin{tikzpicture}[shorten >=1pt, node distance=2cm, on grid, auto, baseline=-1.5cm] 
			\node[state,initial,accepting,minimum size=0.5cm] (q_0) {}; 
			\node[state,accepting,minimum size=0.5cm] (q_1) [below= 1.5cm of q_0] {}; 
			\node[state,minimum size=0.5cm] (q_2) [below= 1.5cm of q_1] {}; 
			\node (q_3) [right= 2.5 cm of q_0] {$\automatonD(j-1)$};
			\path[->] 
				(q_0) edge node [swap] {$s_{j}$} (q_1)
					  edge node {$s_{j-1}$} (q_3)
				(q_1) edge node [swap] {$s_{j-1}$} (q_2);
		\end{tikzpicture}
	}
	\caption{The automata $\automatonU(j)$ (left) and $\automatonD(j)$ (right) defined recursively.}
	\label{fig:automataRecursive}
\end{figure}

\begin{figure}
	\centerline{
		\begin{tikzpicture}[shorten >=1pt, node distance=2cm, on grid, auto]
			\node[state,initial,accepting,minimum size=0.5cm] (hj-1)   {}; 
			\node[state,accepting,minimum size=0.5cm] (ij-1) [below= 1.5cm of hj-1] {}; 
			\node[state,minimum size=0.5cm] (dj-1) [below= 1.5cm of ij-1] {}; 
			\node[state,accepting,minimum size=0.5cm] (hj) [right= 2.5cm of hj-1] {};
			\node[state,accepting,minimum size=0.5cm] (ij) [below= 1.5cm of hj] {}; 
			\node[state,minimum size=0.5cm] (dj) [below= 1.5cm of ij] {}; 
			\node[state,accepting,minimum size=0.5cm] (hj+1) [right= 2.5cm of hj] {};
			\node[state,accepting,minimum size=0.5cm] (ij+1) [below= 1.5cm of hj+1] {}; 
			\node[state,minimum size=0.5cm] (dj+1) [below= 1.5cm of ij+1] {};	  
			\node (void) [right= 2.5cm of hj+1] {\dots};
			\node[state,accepting,minimum size=0.5cm] (hn-2) [right= 2.5cm of void] {};
			\node[state,accepting,minimum size=0.5cm] (in-2) [below= 1.5cm of hn-2] {}; 
			\node[state,minimum size=0.5cm] (dn-2) [below= 1.5cm of in-2] {}; 
			\node[state,accepting,minimum size=0.5cm] (hn-1) [right= 2.5cm of hn-2] {}; 
			\node[state,accepting,minimum size=0.5cm] (in-1) [below= 1.5cm of hn-1] {}; 
			\path[->] 
				(hj-1) edge node [swap] {$s_{j-1}$} (ij-1)
					   edge node {$s_j$} (hj)
				(ij-1) edge node [swap] {$s_j$} (dj-1)
				(hj) edge node [swap] {$s_{j}$} (ij)
					 edge node {$s_{j+1}$} (hj+1)
				(ij) edge node [swap] {$s_{j+1}$} (dj)			
				(hj+1) edge node [swap] {$s_{j+1}$} (ij+1)
					   edge  node {$s_{j+2}$} (void)
				(ij+1) edge  node [swap] {$s_{j+2}$} (dj+1)			  
				(void) edge node {$s_{n-1}$} (hn-2)
				(hn-2) edge node [swap] {$s_{n-2}$} (in-2)
					   edge node {$s_{n-1}$} (hn-1)
				(in-2) edge node [swap] {$s_{n-1}$} (dn-2)
				(hn-1) edge node [swap] {$s_{n-1}$} (in-1);
		\end{tikzpicture}
		}
	\caption{The complete automaton~$\automatonU(j)$.}
	\label{fig:automataComplete}
\end{figure}

\enlargethispage{.3cm}
Consider now arbitrary subsets~$U$ and~$D$ of~$\{2, \dots, n-1\}$.
It follows from \cref{thm:patternAvoidance} that a permutation is minimal in its $(U,D)$-permutree class if and only if it admits a reduced word accepted by~$\automatonU(j)$ for each~$j \in U$ and by~$\automatonD(j)$ for each~$j \in D$.
In general, the reduced words accepted by the automata~$\automatonU(j)$ for each~$j \in U$ and by~$\automatonD(j)$ for each~$j \in D$ are distinct.
We prove however in \cref{sec:intersectionsAutomata} that there is a reduced word simultaneously accepted by all these automata when~$U$ and~$D$ are disjoint.

\begin{theorem}\label{thm:permutreeMinimal}
Consider two disjoint subsets~$U$ and~$D$ of~$\{2, \dots, n-1\}$.
The following conditions are equivalent for~$\pi \in \fS_n$:
\begin{itemize}
	\item $\pi$ admits a reduced word accepted by all automata~$\automatonU(j)$ for~$j \in U$ and~$\automatonD(j)$ for~${j \in D}$,
	\item $\pi$ contains no subword $jki$ if~$j \in U$ and~$kij$ if~$j \in D$ for any~$i < j < k$.
\end{itemize}
\end{theorem}

\cref{thm:permutreeMinimal} implies that given any permutation~$\pi$ avoiding $jki$ if~$j \in U$ and~$kij$ if~$j \in D$, we can sort~$\pi$ while preserving these avoiding conditions.
The resulting sorting procedures, that we call \defn{$(U,D)$-permutree sorting}, are discussed in~\cref{subsec:permutreeSorting}.
For instance, stack sorting is a $(\{2, \dots, n-1\}, \varnothing)$-permutree sorting.

Finally, in the particular situation when the subsets~$U$ and~$D$ form a partition of~$\{2, \dots, n-1\}$, we actually show that the reduced word simultaneously accepted by the automata $\automatonU(j)$ for~${j \in U}$ and~$\automatonD(j)$ for~$j \in D$ is the $c$-sorting word of~$\pi$ as defined in~\cite{Reading-CoxeterSortable}.
This yields in particular an alternative proof that condition~(\ref{condition:sortableCambrian}') characterizes the Cambrian minimal permutations.
This new perspective on $c$-sortability is explored in \cref{sec:coxeterSortable}.


\section{Automata for reduced words}\label{sec:proofPatternAvoidance}

\subsection{Reduced words, automata, and subword avoidance}

We start with properly fixing the few notations needed in this paper.
We consider the \defn{symmetric group}~$\fS_n$ of permutations of the set~$[n] \eqdef \{1,\dots,n\}$.
It is generated by the \defn{simple transpositions}~$s_i \eqdef (i \;\; i+1)$ for $i \in [n-1]$ which are involutions~$s_i^2 = id$ and satisfy the commutation relations $s_i \cdot s_j =s_j \cdot s_i$ if $|i-j|>1$ and the braid relations~$s_i \cdot s_{i+1} \cdot s_i = s_{i+1} \cdot s_i \cdot s_{i+1}$.
Note that we multiply permutations as usual, so that the left multiplication by~$s_i$ exchanges the entries with values~$i$ and~$i+1$, while the right multiplication by~$s_i$ exchanges the entries at positions~$i$ and~$i+1$.
Each permutation~$\pi$ decomposes into products of transpositions of the form~$\pi = s_{i_1} \cdots s_{i_k}$ with~$i_1, \dots, i_k \in [n-1]$.
The minimal number of transpositions in such a decomposition is the \defn{length}~$\ell(\pi)$ of~$\pi$ and the decompositions of length~$\ell(\pi)$ are the \defn{reduced words} for~$\pi$.

Consider now the automata $\automatonU(j)$ and $\automatonD(j)$ described in the introduction, see \cref{fig:automataRecursive,fig:automataComplete,fig:TreePartialOrientations}.
We call a state \defn{healthy}, \defn{ill}, or \defn{dead} depending on whether it belongs to the top, middle, or bottom row of the automata.
Each state has $n-1$ possible transitions, one for each $s_i$ for~$i \in [n-1]$, but we only explicitly indicate the ones between different states.
The automata $\automatonU(j)$ and $\automatonD(j)$ take as entry a reduced word~$s_{i_1} \cdots s_{i_\ell}$ for a permutation of~$\fS_n$ and read it from left to right.
We start at the initial state (marked with ``start''), and at step~$t$ we follow the transition marked by the letter~$s_{i_t}$ if any, or stay in the current state otherwise.
After~$\ell$ steps, the reduced word~$s_{i_1} \cdots s_{i_\ell}$ is declared accepted if the current state is accepting (doubly circled, healthy or ill states), and rejected otherwise (dead states).

For a fixed~$j \in \{2, \dots, n-1\}$, we say that a permutation~$\pi$ \defn{avoids} $jki$ (resp.~$kij$) if for any~$i < j < k$, the word~$jki$ (resp.~$kij$) does not appear as a subword of the one-line notation of~$\pi$, or said differently if there are no positions~$p < q < r$ such that~$\pi(r) < \pi(p) = j < \pi(q)$ (resp.~$\pi(q) < \pi(r) = j < \pi(p)$).
We insist on the fact that while the value~$j$ is fixed, $i$ and~$k$ take all possible values such that~$1 \le i < j < k \le n$.
This convenient notion here should not be mixed up with the notion of pattern avoidance where~$j$ is not fixed.
For instance, a permutation avoids the pattern~$231$ if and only if it avoids $jki$ for all~$j \in \{2, \dots, n-1\}$.

\begin{example}\label{exm:patternAvoidance}
	The permutation~$42135$ avoids~$2ki$, $3ki$, and~$4ki$ (and therefore the pattern~$231$), but contains~$ki3$ (and therefore the pattern~$312$) because its one-line notation contains~$423$. 
\end{example}

\subsection{Behavior under left multiplication}

In the perspective of proving \cref{thm:patternAvoidance}, we study the two properties ``$\pi$ admits a reduced word accepted by~$\automatonU(j)$ (resp.~$\automatonD(j)$)'' and ``$\pi$ avoids $jki$ (resp.~$kij$)''.
In this section, we study the behavior of these properties under left multiplication.
We treat separately the cases when we multiply by a permutation commuting with both~$s_{j-1}$ and~$s_j$ (\cref{lem:vincent2}), by $s_{j-1}$ (\cref{lem:vincent3}), and by~$s_j$ (\cref{lem:vincent4}).

\begin{lemma}\label{lem:vincent2}
If two permutations~$\sigma, \tau \in \fS_n$ are such that $\sigma([j-1]) = [j-1]$, $\sigma(j) = j$, and $\sigma([n]\ssm [j]) = [n]\ssm [j]$ and $\length(\sigma \cdot \tau) = \length(\sigma) + \length(\tau)$, then:
\begin{enumerate}
	\item $\tau$ admits a reduced word accepted by $\automatonU(j)$ (resp.~$\automatonD(j)$) if and only if $\sigma \cdot \tau$ admits a reduced word accepted by $\automatonU(j)$ (resp.~$\automatonD(j)$),
	\item $\tau$ avoids $jki$ (resp.~$kij$) if and only if $\sigma \cdot \tau$ avoids $jki$ (resp.~$kij$).
\end{enumerate}
\end{lemma}

\begin{proof}
We deal with the two statements separately:
\begin{enumerate} 
	\item The conditions on~$\sigma$ imply that none of its reduced words contain the transpositions $s_{j-1}$ or $s_j$. Therefore, while reading any reduced word for~$\sigma$, the automaton~$\automatonU(j)$ stays in the initial state. The result immediately follows.
	\item Since~$\sigma$ permutes only values smaller than $j$ between themselves and values greater than $j$ between themselves, we see a subword~$jki$ with~$i < j < k$ in~$\tau$ if and only if we see a subword~$jk'i'$ with~$i' < j < k'$ in~$\sigma \cdot \tau$, where~$i' = \sigma(i)$ and~$k' = \sigma(k)$.
	\qedhere
\end{enumerate}
\end{proof}

\begin{example}\label{exm:lemVincent2}
Consider~$j \eqdef 4$ and the permutations~$\sigma \eqdef 312465 = s_2 \cdot s_1 \cdot s_5$, $\tau_1 \eqdef 143256 = s_3 \cdot s_2 \cdot s_3$, and~$\tau_2 \eqdef 124536 = s_3 \cdot s_4$.
Multiplying we obtain~$\sigma\cdot\tau_1=342165$ and~$\sigma\cdot\tau_2 = 314625$. Observe that
\begin{enumerate}
	\item $\automatonU(4)$ accepts all reduced words of both~$\tau_1$ and~$\sigma\cdot \tau_1$ on its first ill state, and rejects all reduced words of both~$\tau_2$ and~$\sigma\cdot \tau_2$,
	\item both~$\tau_1$ and~$\sigma\cdot\tau_1$ avoid~$4ki$, while both~$\tau_2$ and~$\sigma\cdot\tau_2$ contain~$4ki$.
\end{enumerate}
\end{example}

\begin{lemma}\label{lem:vincent3}
If a permutation $\tau \in \fS_n$ has a reduced word starting with $s_{j-1}$ (resp.~$s_j$) and accepted by $\automatonU(j)$ (resp.~$\automatonD(j)$), then
\begin{enumerate} 
	\item $\tau$ does not permute $j$ and $j+1$ (resp.~$j-1$ and $j$),
	\item $\tau$ avoids $jki$ (resp.~$kij$).
\end{enumerate}
\end{lemma}

\begin{proof}
Consider a reduced word $w$ starting with $s_{j-1}$ and accepted by~$\automatonU(j)$. 
We deal with the two statements separately:
\begin{enumerate} 
	\item Since $w$ starts with~$s_{j-1}$, the values~$j-1$ and~$j$ are reversed in~$\tau$. If~$j$ and~$j+1$ were also reversed in~$\tau$, we would obtain that~$j-1$ and~$j+1$ are reversed. It follows that $w$ must contain a~$s_j$ at some point after the~$s_{j-1}$. But this would lead to a dead state, contradicting the assumption that~$w$ is accepted.
	\item Let $\tau = s_{j-1}\cdot \rho$. Since any reduced word of $\rho$ cannot contain~$s_j$ (as $w$ is accepted by $\automatonU(j)$), we have that $\rho([j]) = [j]$ and $\rho([n+1] \ssm [j]) = [n+1] \ssm [j]$ and find that $\rho$ contains no subword~$ki$ with $i < j < k$. Therefore, $\tau$ avoids $jki$.
	\qedhere
\end{enumerate}
\end{proof}

\begin{example}\label{exm:lemVincent3}
Consider~$j \eqdef 4$ and the permutation~$\tau \eqdef 413265$, whose reduced word~$s_3 \cdot s_5 \cdot s_2 \cdot s_1 \cdot s_3$ is accepted by~$\automatonU(4)$.
Observe that
\begin{enumerate}
	\item $\tau$ indeed does not permute the values~$4$ and~$5$,
	\item $\tau$ avoids~$4ki$.
\end{enumerate}
\end{example}

\begin{lemma}\label{lem:vincent4}
If a permutation~$\tau \in \fS_n$ does not permute $j$ and $j+1$ (resp.~$j-1$ and~$j$), then
\begin{enumerate} 
	\item $s_j \cdot \tau$ (resp.~$s_{j-1} \cdot \tau$) admits a reduced word accepted by $\automatonU(j)$ (resp.~$\automatonD(j)$) if and only if $\tau$ admits a reduced word accepted by $\automatonU(j+1)$ (resp.~$\automatonD(j-1)$),
	\item $s_j \cdot \tau$ (resp.~$s_{j-1} \cdot \tau$) avoids $jki$ (resp.~$kij$) if and only if $\tau$ avoids $(j+1)ki$ (resp.~$ki(j-1)$).
\end{enumerate}
\end{lemma}

\begin{proof}
We deal with the two statements separately:
\begin{enumerate} 
	\item Suppose that $w$ is a reduced word for $\tau$ accepted by $\automatonU(j+1)$. Since $\tau$ does not permute $j$ and $j+1$, we know that $s_j \cdot w$ is a reduced word for~$s_j \cdot \tau$, and it is accepted by~$\automatonU(j)$ by construction. Conversely assume that $s_j \cdot \tau$ admits a reduced word $w$ accepted by $\automatonU(j)$. Since $s_j \cdot \tau$ permutes $j$ and $j+1$, $w$ must contain a $s_j$ and cannot start by $s_{j-1}$ by \cref{lem:vincent3}. Due to \cref{lem:vincent2} we can also assume that $w$ starts with $s_j$. Thus the suffix is a reduced word for $\tau$ that is accepted by $\automatonU(j+1)$.
	\item Observe that since~$j$ and~$j+1$ are reversed in~$s_j \cdot \tau$ and not in~$\tau$, the value $j+1$ cannot serve as~$k$ in a subword~$jki$ of~$s_j \cdot \tau$ and the value~$j$ cannot serve as~$i$ in a subword $(j+1)ki$ in~$\tau$. The result thus immediately follow from the fact that the left multiplication by~$s_j$ only exchanges the values $j$ and $j+1$.
	\qedhere
\end{enumerate}
\end{proof}

\begin{example}\label{exm:lemVincent4}
Consider~$j \eqdef 4$ and the permutations~$\tau_1 \eqdef 142536$ and~$\tau_2 \eqdef 142563$ that do not permute~$4$ and~$5$. 
Multiplying we obtain~$s_4 \cdot \tau_1 = 152436$ and~$s_4 \cdot \tau_2 = 152463$. Observe that
\begin{enumerate}
	\item the reduced word~$s_4 \cdot s_3 \cdot s_4 \cdot s_2$ of~$s_4 \cdot \tau_1$ is accepted by~$\automatonU(4)$ and the reduced word~$s_3 \cdot s_4 \cdot s_2$ of~$\tau_1$ is accepted by~$\automatonU(5)$, while all reduced words of~$s_4 \cdot \tau_2$ are rejected by~$\automatonU(4)$ and all reduced words of~$\tau_2$ are rejected by~$\automatonU(5)$,
	\item $s_4 \cdot \tau_1$ avoids~$4ki$ and $\tau_1$ avoids~$5ki$, while $s_4 \cdot \tau_2$ contains~$463$ and $\tau_2$ contains~$563$.
\end{enumerate}
\end{example}

\subsection{Proof of \cref{thm:patternAvoidance}}

With these lemmas in hand, we are now ready to show \cref{thm:patternAvoidance} that we repeat here for convenience.

\newtheorem*{thm:patternAvoidance}{\cref{thm:patternAvoidance}}
\begin{thm:patternAvoidance}
Fix $j \in \{2, \dots, n-1\}$.
The following conditions are equivalent for~$\pi \in \fS_n$:
\begin{itemize}
	\item $\pi$ admits a reduced word accepted by the automaton~$\automatonU(j)$ (resp.~$\automatonD(j)$),
	\item $\pi$ contains no subword $jki$ (resp.~$kij$) with~$i < j < k$.
\end{itemize}
\end{thm:patternAvoidance}

\begin{proof}[Proof of \cref{thm:patternAvoidance}]
We work by induction on the length of the permutations.
Assume that a permutation $\pi$ admits a reduced word accepted by $\automatonU(j)$. Let $s_i$ be the first letter of this reduced word and let $\tau$ be such that $\pi = s_i \cdot \tau$. We distinguish three cases:\begin{itemize}
	\item If $i = j-1$, then $\pi$ avoids $jki$ by \cref{lem:vincent3}\,(2).
	\item If $i = j$, then $\tau$ admits a reduced word accepted by $\automatonU(j+1)$ by \cref{lem:vincent4}\,(1). We obtain by induction that $\tau$ avoids $(j+1)ki$. Thus $\pi = s_j \cdot \tau$ avoids $jki$ by \cref{lem:vincent4}\,(2).
	\item Otherwise, $\tau$ admits a reduced word accepted by $\automatonU(j)$ by \cref{lem:vincent2}\,(1), so that $\tau$ avoids $jki$ by induction. Thus $\pi = s_i \cdot \tau$ avoids $jki$ by \cref{lem:vincent2}\,(2).
\end{itemize}
In all three cases, we proved that $\pi$ avoids $jki$.

Assume now that a permutation $\pi$ avoids $jki$. Here, we have to be careful because not all reduced words for $\pi$ will be accepted by $\automatonU(j)$ \apriori. So we have to construct a good reduced word for~$\pi$. We distinguish two cases:
\begin{itemize}
	\item Assume first that there is $m > j$ such that $\pi$ reverses $j$ and $m$, and pick $m$ minimal for this property. It follows that $\pi$ reverses $\ell$ and $m$ for all $\ell$ in $\{j,\dots,m-1\}$. In other words, $\pi$ admits a reduced word starting by the cyclic permutation $(j, j+1, ..., m) = s_{m-1} \cdot s_{m-2} \cdots s_{j+1} \cdot s_j$. Define $\sigma = s_{m-1} \cdot s_{m-2}\cdots s_{j+1}$ and $\tau$ such that $\pi = \sigma \cdot s_j \cdot \tau$ and so that this word is reduced. By Lemmas \ref{lem:vincent2}\,(2) and \ref{lem:vincent4}\,(2), $\tau$ avoids $(j+1)ki$. By induction, we obtain that it admits a reduced word accepted by $\automatonU(j+1)$. By Lemmas \ref{lem:vincent2}\,(1) and \ref{lem:vincent4}\,(1), we conclude that $\pi$ admits a reduced word accepted by $\automatonU(j)$.
	\item Assume now that $j$ appears before all $m > j$ in $\pi$. Consider any reduced word for $\pi$. If this word is accepted by $\automatonU(j)$, we are done. Otherwise, it first uses $s_{j-1}$ and then $s_j$ (otherwise, $j$ and some $m > j$ would be exchanged). Call $i$ and $k$ the two elements that are exchanged when the reduced word first uses $s_j$. We have $i < j < k$ and $jki$ in $\pi$ (because $j$ and $k$ are not exchanged in $\pi$, and $i$ and $k$ are already exchanged so they will remain exchanged in $\pi$), a contradiction.
\end{itemize}
In both cases, we proved that $\pi$ admits a reduced word accepted by $\automatonU(j)$.
\end{proof}


\section{Structure of accepted reduced words}\label{sec:algorithmicCombinatorialConsequences}

In this section, we explore some additional properties of the set of reduced words accepted by the automata~$\automatonU(j)$ and~$\automatonD(j)$ and derive relevant algorithmic and combinatorial consequences.

\subsection{The set of accepted reduced words}\label{subsec:principles}

Observe that a given permutation~$\pi$ may admit both accepted and rejected reduced words.
For instance, the (non-simple) transposition~${(j-1 \;\; j+1)}$ has reduced words~$s_j \cdot s_{j-1} \cdot s_j$ accepted by~$\automatonU(j)$ and~$s_{j-1} \cdot s_j \cdot s_{j-1}$ rejected by~$\automatonU(j)$.
However, \cref{prop:prefixesReducedExpressions,prop:algorithm,prop:sameStateAcceptedReducedExpressions} below show that the set of accepted reduced words satisfies the following three principles:
\begin{itemize}
	\item \textbf{Who can do more can do less!} --- The set of accepted reduced words is closed by~prefix.
	\item \textbf{When health goes, everything goes!} --- If $\pi$ admits an accepted reduced word, then $\pi$ admits an accepted reduced word starting with any descent that remains in the healthy states.
	\item \textbf{All roads lead to Rome!} --- All accepted reduced words for~$\pi$ end at the same~state.
\end{itemize}

\begin{proposition}\label{prop:prefixesReducedExpressions}
The set of reduced words accepted by~$\automatonU(j)$ (resp.~$\automatonD(j)$) is closed by prefix.
\end{proposition}

\begin{proof}
This immediately follows from the fact that the set of reduced words is closed by prefix, and that the set of accepting states of~$\automatonU(j)$ is connected and contains the initial state.
\end{proof}

\begin{proposition}\label{prop:algorithm}
Let $\ell \in [n-1]$ be distinct from $j-1$ (resp.~$j$).
A permutation $\pi \in \fS_n$ that avoids $jki$ (resp.~$kij$) and reverses $\ell$ and $\ell+1$ admits a reduced word starting with $s_\ell$ and accepted by~$\automatonU(j)$ (resp.~$\automatonD(j)$).
\end{proposition}

\begin{proof}
Since $\pi$ reverses $\ell$ and $\ell+1$, it admits a reduced word of the form~$\pi = s_\ell \cdot \tau$. Now consider two cases depending on the value of $\ell$:
\begin{itemize}
	\item If $\ell = j$, then $\tau$ avoids $(j+1)ki$ by \cref{lem:vincent4}\,(2). Hence, $\tau$ has a reduced word accepted by $\automatonU(j+1)$ by \cref{thm:patternAvoidance}, and we conclude by \cref{lem:vincent4}\,(1).
	\item Otherwise, $\ell$ is neither $j-1$ nor $j$, so that $\tau$ avoids $jki$ by \cref{lem:vincent2}\,(2). Hence, $\tau$ has a reduced word accepted by $\automatonU(j)$ by \cref{thm:patternAvoidance}, and we conclude by \cref{lem:vincent2}\,(1).
	\qedhere
\end{itemize}
\end{proof}

\begin{proposition}\label{prop:sameStateAcceptedReducedExpressions}
Given a permutation $\pi \in \fS_n$, all the reduced words for~$\pi$ accepted by $\automatonU(j)$ (resp.~$\automatonD(j)$) end at the same state.
\end{proposition}

To prove \cref{prop:sameStateAcceptedReducedExpressions}, it would be enough to check that any two reduced words accepted by~$\automatonU(j)$ that differ by a single commutation or a single braid relation indeed end at the same state.
However, we prefer to prove instead the following stronger but more technical version of \cref{prop:sameStateAcceptedReducedExpressions}.

\begin{proposition}\label{prop:sameStateAcceptedReducedExpressionsRefined}
For a permutation~$\pi \in \fS_n$, let~$\ninv^j(\pi) = |\set{(j,i)}{i < j \text{ and } \pi^{-1}(i) > \pi^{-1}(j)}|$ and~$\ninv_j(\pi) = |\set{(k,j)}{j < k \text{ and } \pi^{-1}(j) > \pi^{-1}(k)}|$ denote the numbers of inversions involving~$j$ as the first and second entry respectively. Then
\begin{enumerate}[(i)]
	\item if $\ninv^j(\pi) = 0$, then all reduced words for~$\pi$ end at the same healthy state of~$\automatonU(j)$,
	\item if~$\ninv_j(\pi) = 0$, then all reduced words for~$\pi$ end at the same state of~$\automatonU(j)$, which might be healthy if $\pi$ avoids~$ji$, ill if~$\pi$ contains~$ji$ but avoids~$jki$, or dead if~$\pi$ contains~$jki$,
	\item if $\ninv^j(\pi) \ne 0 \ne \ninv_j(\pi)$, all accepted reduced words for~$\pi$ end at the same ill state of~$\automatonU(j)$ while the rejected reduced words may end at distinct dead states of~$\automatonU(j)$.
\end{enumerate}
Moreover, all reduced words for~$\pi$ accepted by~$\automatonU(j)$ end in the $(\ninv_j(\pi)+1)$st column of~$\automatonU(j)$.
A similar statement holds for~$\automatonD(j)$ by exchanging $\ninv_j(\pi)$ and~$\ninv^j(\pi)$.
\end{proposition}

\begin{proof}
The proof works by induction on the length of~$\pi$.
Consider an arbitrary reduced word~$w$ for~$\pi$, starting with a transposition~$s_\ell$, and write~$w = s_\ell \cdot w'$ and~$\pi = s_\ell \cdot \tau$.
Observe~that:
\begin{itemize}
	\item if $\ell \notin \{j-1, j\}$, then~$s_\ell$ loops in~$\automatonU(j)$, $\ninv^j(\pi) = \ninv^j(\tau)$ and $\ninv_j(\pi) = \ninv_j(\tau)$,
	\item if~$\ell = j$, then~$s_j$ goes to~$\automatonU(j+1)$, $\ninv^j(\pi) = \ninv^{j+1}(\tau)$ and~${\ninv_j(\pi) = \ninv_{j+1}(\tau) + 1}$,
	\item if~$\ell = j-1$, then~$s_{j-1}$ goes to the first ill state of~$\automatonU(j)$, $\ninv^j(\pi) = \ninv^{j+1}(\tau) + 1$ and~$\ninv_j(\pi) = \ninv_{j+1}(\tau)$.
\end{itemize}
By induction, we obtain that the reduced word~$w'$ for~$\tau$ ends as predicted in the statement.
The previous observations ensure that the reduced word~$w$ for~$\pi$ also does.
\end{proof}

\begin{example}\label{exm:sameStateAcceptedReducedExpressionsRefined}
We present an example of each case:
\begin{enumerate}[(i)]
	\item For~$\pi \eqdef 4312$, we have that~$\ninv^2(\pi)=0$ and all of its~$5$ reduced words end at the third healthy state of $\automatonU(2)$.
	\item For~$\pi \eqdef 32145$ (resp.~$\pi \eqdef 43215$, resp.~$\pi \eqdef 43251$), we have~$\ninv_4(\pi) = 0$ and all its~$2$ (resp.~$16$, resp.~$35$) reduced words end at the first healthy (resp.~ill, resp.~dead) state~of~$\automatonU(4)$.
	\item For~$\pi \eqdef 4321$, we have~$\ninv^2(\pi) = |\{(2,1)\}| = 1$ and $\ninv_2 = |\{(3,2),(4,2)\}| = 2$. Among the~$16$ reduced words of~$\pi$, the automaton~$\automatonU(2)$ accepts~$7$ at its third ill state, rejects~$7$ at its first dead state, and rejects the other~$2$ at its second dead state.
\end{enumerate}
\end{example}

\subsection{Finding accepted reduced words}

\cref{prop:algorithm} has a strong algorithmic consequence.
Imagine we want to test whether a permutation~$\pi \in \fS_n$ is minimal in its permutree class for~$U = \{j\}$ and~$D = \varnothing$.
Of course, the quickest way is to check for all~$i < j < k$ whether~$\pi$ contains the subword~$jki$.
But since this interpretation will be lost beyond type~$A$, let us impose the use of reduced words for~$\pi$ to make this test.
While it would be \apriori{} necessary to check all reduced words on the automaton~$\automatonU(j)$, \cref{prop:algorithm} enables us to construct without loss of generality a candidate reduced word for~$\pi$ and we will just need to check that this one is accepted by~$\automatonU(j)$.
Somewhat dually, one can also construct a reduced word accepted by~$\automatonU(j)$ that is a reduced word for~$\pi$ if and only if~$\pi$ avoids~$jki$.
This is done in the following algorithm, that we call \mbox{\defn{$(\{j\}, \varnothing)$-permutree sorting}}.
The reader is invited to write down the symmetric $(\varnothing, \{j\})$-permutree sorting.
We will discuss further permutree sorting in \cref{subsec:permutreeSorting}.

\bigskip
\IncMargin{1em}
\SetKwInOut{Input}{Input}\SetKwInOut{Output}{Output}
\SetKwFor{Repeat}{repeat}{}{}
\SetKwIF{If}{ElseIf}{Else}{if}{then}{else if}{else}{}
\DontPrintSemicolon
\begin{algorithm}[H]
	\Input{a permutation $\pi \in \fS_n$ and an integer~$j \in [n]$}
	\Output{a reduced word accepted by~$\automatonU(j)$, candidate reduced word for~$\pi$}
	$w \eqdef \varepsilon$ \;
	\Repeat{}{
		\If{$\exists \; \ell \ne j-1$ such that $\ell$ and $\ell+1$ are reversed in~$\pi$}{
			$\pi \eqdef s_\ell \cdot \pi$, \quad $w \eqdef w \cdot s_\ell$ \;
			\lIf{$\ell = j$}{
				$j \eqdef j+1$
			}
		}
	}
	\If{$j-1$ and $j$ are reversed in~$\pi$}{
		$\pi \eqdef s_{j-1} \cdot \pi$, \quad $w \eqdef w \cdot s_{j-1}$ \;
		$w \eqdef w \cdot w' \cdot w''$ where~$w'$ sorts~$\pi_{[j]}$ and $w''$ sorts $\pi_{[n] \ssm [j]}$ \;
	}
	\Return $w$
	\caption{$(\{j\}, \varnothing)$-permutree sorting}
	\label{algo:shortcutsUj}
\end{algorithm}
\bigskip

\begin{example}\label{exm:algo1}
	Let us present the $(\{2\},\varnothing)$-permutree sorting algorithm in action for the permutations~$\pi_1 \eqdef 3421$ and~$\pi_2 \eqdef 4231$. The steps of the algorithm are presented in \cref{tab:algo1}. Each row contains the states of the permutation~$\pi$ and of the word~$w$ and the current values of~$j$ and~$\ell$ in use at each step. Notice that for~$\pi_1 \eqdef 3421$ the algorithm ends with the identity, which coincides with the fact that~$\pi_1 \eqdef 3421$ avoids~$2ki$. In contrast, for~$\pi_2 \eqdef 4231$ the algorithm ends with the permutation~$1243$, meaning that~$\pi_2$ is not~$(\{2\},\varnothing)$-sortable, which coincides with the fact that~$\pi_2 \eqdef 4231$ contains~$2ki$. 

	\begin{table}[h!]
		\centerline{
		\begin{tabular}[t]{l|l|c|c}
			$\pi_1$ & $w_1$ & $j_1$ & $\ell_1$ \\
			\hline
			$3421$ & $\varepsilon$ & $2$ & $2$ \\
			$2431$ & $s_2$ & $3$ & $1$ \\
			$1432$ & $s_2 \cdot s_1$ & $3$ & $3$ \\
			$1342$ & $s_2 \cdot s_1 \cdot s_3$ & $4$ & $2$ \\
			$1243$ & $s_2 \cdot s_1 \cdot s_3 \cdot s_2$ & $4$ & $3$ \\
			$1234$ & $s_2 \cdot s_1 \cdot s_3 \cdot s_2 \cdot s_3$ & $4$ &
		\end{tabular}
		\qquad\qquad
		\begin{tabular}[t]{l|l|c|c}
			$\pi_2$ & $w_2$ & $j_2$ & $\ell_2$ \\
			\hline
			$4231$ & $\varepsilon$ & $2$ & $3$ \\
			$3241$ & $s_3$ & $2$ & $2$ \\
			$2341$ & $s_3 \cdot s_2$ & $3$ & $1$ \\
			$1342$ & $s_3 \cdot s_2 \cdot s_1$ & $3$ & $2$ \\
			$1243$ & $s_3 \cdot s_2 \cdot s_1 \cdot s_2$ & $3$ &
		\end{tabular}
		}
		\caption{$(\{2\},\varnothing)$-permutree sorting of~$\pi_1 \eqdef 3421$ and~$\pi_2 \eqdef 4231$.}
		\label{tab:algo1}
	\end{table}
\end{example}               

\begin{corollary}\label{coro:algorithm}
For any permutation~$\pi$ and~$j \in \{2, \dots, n-1\}$, \cref{algo:shortcutsUj} returns a reduced~word~$w$ accepted by $\automatonU(j)$ with the property that~$w$ is a reduced word for~$\pi$ if and only if~$\pi$ avoids~$jki$.
\end{corollary}

\begin{proof}
This algorithm constructs a candidate reduced word for $\pi$ while following the automaton~$\automatonU(j)$ and prioritizing healthy states over ill states.
Lines 2 to 5 start by all possible transitions~$s_\ell$ that remain in healthy states, updating~$j$ to~$j+1$ when~$\ell = j$ according to \cref{lem:vincent4} (if condition at line 5).
When we have exhausted all these transitions, if we need to go to an ill state (if condition at line 6) applying~$s_{j-1}$, then we are not anymore allowed to use~$s_j$ and we obtain a candidate reduced word by sorting independently the first $j$ positions of~$\pi$ with a reduced word in~$\{s_1, \dots, s_{j-1}\}^*$ and the last $[n] \ssm [j]$ positions of~$\pi$ with a reduced word in~$\{s_{j+1}, \dots, s_{n-1}\}^*$.
The resulting reduced word~$w$ is clearly accepted by~$\automatonU(j)$ because we never allow the transition from an ill state to the corresponding dead state.
If~$w$ is a reduced word for~$\pi$, then $\pi$ avoids~$jki$ by \cref{thm:patternAvoidance}.
Conversely, if~$\pi$ avoids~$jki$, then $w$ must be a reduced word for~$\pi$ since the choice to start with~$s_\ell$ is valid in lines 2 to 5 by \cref{prop:algorithm} and forced in lines 6 to 8 (since all reduced words of~$\pi$ then start by~$s_\ell$).
\end{proof}

\begin{remark}
We really wrote \cref{algo:shortcutsUj} as a sorting algorithm.
It first tries to sort the permutation~$\pi \in \fS_n$ while avoiding to swap~$j-1$ and~$j$ for a certain token~$j$ (and changing the token when swapping~$j$ and~$j+1$).
Once it is forced to swap~$j-1$ and~$j$, it tries to sort the permutation~$\pi$ while avoiding to swap any value of~$[j]$ with a value of~$[n] \ssm [j]$.
If we were only interested in deciding whether the permutation~$\pi$ is $(\{j\}, \varnothing)$-sortable, then we could stop and accept the permutation as soon as we reach $j = n$, and we could just check at line 8 of the algorithm whether~$\pi([j]) = [j]$ and~$\pi([n] \ssm [j]) = [n] \ssm [j]$.
\end{remark}

\subsection{Generating trees on accepted reduced words}\label{subsec:trees}

\cref{prop:prefixesReducedExpressions,prop:sameStateAcceptedReducedExpressions} also have a relevant consequence, more combinatorial this time.
Namely, they naturally define generating trees for the $(\{j\}, \varnothing)$-permutree minimal permutations, following certain special reduced words for them.
To construct these trees, pick an arbitrary priority order~$\prec$ on~$\{s_1, \dots, s_{n-1}\}$.
For a $(\{j\}, \varnothing)$-permutree minimal permutation~$\pi \in \fS_n$, denote by~$\pi(\{j\}, \varnothing, \prec)$ the $\prec$-lexicographic minimal reduced word for~$\pi$ that is accepted by~$\automatonU(j)$.
Denote by~$\lexmin(n, \{j\}, \varnothing, \prec)$ the set of reduced words of the form~$\pi(\{j\}, \varnothing, \prec)$ for all $(\{j\}, \varnothing)$-permutree minimal permutations~$\pi \in \fS_n$.
The following statement is an analogue of \cref{prop:prefixesReducedExpressions}.

\begin{proposition}\label{prop:generatingTrees}
The set $\lexmin(n, \{j\}, \varnothing, \prec)$ is closed by prefix.
\end{proposition}

\begin{proof}
Consider a reduced word~$w = u \cdot v$ where~$u$ is not in~$\lexmin(n, \{j\}, \varnothing, \prec)$.
If~$u$ is not accepted by~$\automatonU(j)$, neither is~$w$ by \cref{prop:prefixesReducedExpressions}.
Otherwise, there exists a reduced word~$u'$ representing the same permutation as~$u$, accepted by~$\automatonU(j)$ and $\prec$-lexicographic smaller than~$u$.
By \cref{prop:sameStateAcceptedReducedExpressions}, the reduced words~$u$ and~$u'$ end at the same state of~$\automatonU(j)$.
Therefore, if~$w = u \cdot v$ is accepted by~$\automatonU(j)$, so is~$u' \cdot v$.
Since~$u' \cdot v$ is $\prec$-lexicographically smaller than~$u \cdot v$ and represents the same permutation, this ensures that~$w$ is not in~$\lexmin(n, \{j\}, \varnothing, \prec)$.
\end{proof}

\cref{prop:generatingTrees} yields a natural generating tree for $\lexmin(n, \{j\}, \varnothing, \prec)$ where the parent of a reduced word~$w$ is obtained by deleting its last letter.
Replacing each reduced word by the corresponding permutation, this provides a generating tree for the $(\{j\}, \varnothing)$-permutree minimal permutations of~$\fS_n$.
Of course there is a similar generating tree for the $(\varnothing, \{j\})$-permutree minimal permutations of~$\fS_n$.
\cref{fig:TreePartialOrientations} presents these generating trees for~$n = 4$ and~$j = 2, 3$, with the priority order~${s_1 \prec s_2 \prec s_3}$.
It is natural to draw these trees on top of the Hasse diagram of the right weak order on permutations, defined by inclusion of inversion sets.
In other words, the cover relations in weak order correspond to the swap of the values at two consecutive positions in a permutation, \ie to a right multiplication by a simple transposition.
The edges of the trees corresponding to the right multiplications by $s_1$, $s_2$ and $s_3$ are colored by blue, red, and green respectively.

\hvFloat[floatPos=p, capWidth=h, capPos=r, capAngle=90, objectAngle=90, capVPos=c, objectPos=c]{figure}
{
	\begin{tabular}{cccc}
		$\automatonU(2)$
		&
		$\automatonD(2)$
		&
		$\automatonU(3)$
		&
		$\automatonD(3)$
		\\[.5cm]
		\begin{tikzpicture}[shorten >=1pt, node distance=2cm, on grid, auto]
			\node[state,initial,accepting,minimum size=0.5cm] (hj)   {}; 
			\node[state,accepting,minimum size=0.5cm] (ij) [below= 1.5cm of hj] {}; 
			\node[state,minimum size=0.5cm] (dj) [below= 1.5cm of ij] {}; 
			\node[state,accepting,minimum size=0.5cm] (hj+1) [right= 1.5cm of hj] {};
			\node[state,accepting,minimum size=0.5cm] (ij+1) [below= 1.5cm of hj+1] {}; 
			\node[state,minimum size=0.5cm] (dj+1) [below= 1.5cm of ij+1] {}; 
			\node[state,accepting,minimum size=0.5cm] (hj+2) [right= 1.5cm of hj+1] {};
			\node[state,accepting,minimum size=0.5cm] (ij+2) [below= 1.5cm of hj+2] {}; 
			\path[->] 
				(hj) edge node [swap] {$s_1$} (ij)
					   edge node {$s_2$} (hj+1)
				(ij) edge node [swap] {$s_2$} (dj)
				(hj+1) edge node [swap] {$s_2$} (ij+1)
					 edge node {$s_3$} (hj+2)
				(ij+1) edge node [swap] {$s_3$} (dj+1)
				(hj+2) edge node [swap] {$s_3$} (ij+2);
		\end{tikzpicture}
		&
		\begin{tikzpicture}[shorten >=1pt, node distance=2cm, on grid, auto]
			\node[state,initial,accepting,minimum size=0.5cm] (hj)   {}; 
			\node[state,accepting,minimum size=0.5cm] (ij) [below= 1.5cm of hj] {}; 
			\node[state,minimum size=0.5cm] (dj) [below= 1.5cm of ij] {}; 
			\node[state,accepting,minimum size=0.5cm] (hj-1) [right= 1.5cm of hj] {};
			\node[state,accepting,minimum size=0.5cm] (ij-1) [below= 1.5cm of hj-1] {}; 
			\path[->] 
				(hj) edge node [swap] {$s_2$} (ij)
					   edge node {$s_1$} (hj-1)
				(ij) edge node [swap] {$s_1$} (dj)
				(hj-1) edge node [swap] {$s_1$} (ij-1);
		\end{tikzpicture}
		&
		\begin{tikzpicture}[shorten >=1pt, node distance=2cm, on grid, auto]
			\node[state,initial,accepting,minimum size=0.5cm] (hj)   {}; 
			\node[state,accepting,minimum size=0.5cm] (ij) [below= 1.5cm of hj] {}; 
			\node[state,minimum size=0.5cm] (dj) [below= 1.5cm of ij] {}; 
			\node[state,accepting,minimum size=0.5cm] (hj+1) [right= 1.5cm of hj] {};
			\node[state,accepting,minimum size=0.5cm] (ij+1) [below= 1.5cm of hj+1] {}; 
			\path[->] 
				(hj) edge node [swap] {$s_2$} (ij)
					   edge node {$s_3$} (hj+1)
				(ij) edge node [swap] {$s_3$} (dj)
				(hj+1) edge node [swap] {$s_3$} (ij+1);
		\end{tikzpicture}
		&
		\begin{tikzpicture}[shorten >=1pt, node distance=2cm, on grid, auto]
			\node[state,initial,accepting,minimum size=0.5cm] (hj)   {}; 
			\node[state,accepting,minimum size=0.5cm] (ij) [below= 1.5cm of hj] {}; 
			\node[state,minimum size=0.5cm] (dj) [below= 1.5cm of ij] {}; 
			\node[state,accepting,minimum size=0.5cm] (hj-1) [right= 1.5cm of hj] {};
			\node[state,accepting,minimum size=0.5cm] (ij-1) [below= 1.5cm of hj-1] {}; 
			\node[state,minimum size=0.5cm] (dj-1) [below= 1.5cm of ij-1] {}; 
			\node[state,accepting,minimum size=0.5cm] (hj-2) [right= 1.5cm of hj-1] {};
			\node[state,accepting,minimum size=0.5cm] (ij-2) [below= 1.5cm of hj-2] {}; 
			\path[->] 
				(hj) edge node [swap] {$s_3$} (ij)
					   edge node {$s_2$} (hj-1)
				(ij) edge node [swap] {$s_2$} (dj)
				(hj-1) edge node [swap] {$s_2$} (ij-1)
					 edge node {$s_1$} (hj-2)
				(ij-1) edge node [swap] {$s_1$} (dj-1)
				(hj-2) edge node [swap] {$s_1$} (ij-2);
		\end{tikzpicture}
		\\[.5cm]
    	\begin{tikzpicture}[xscale=.9, yscale=0.7, color=lightgray]
    		\node[blue](P1234) at (0,0){1234};
    		\node[blue](P2134) at (-1,1.5){2134};
    		\node[blue](P1324) at ( 0,1.5){1324};
    		\node[blue](P1243) at ( 1,1.5){1243};
    		\node(P2314) at (-2,3){2314};
    		\node[blue](P3124) at (-1,3){3124};
    		\node[blue](P2143) at ( 0,3){2143};
    		\node[blue](P1342) at ( 1,3){1342};
    		\node[blue](P1423) at ( 2,3){1423};
    		\node[blue](P3214) at (-2.5,4.5){3214};
    		\node(P2341) at (-1.5,4.5){2341};
    		\node[blue](P3142) at (-0.5,4.5){3142};
    		\node(P2413) at ( 0.5,4.5){2413};
    		\node[blue](P4123) at ( 1.5,4.5){4123};
    		\node[blue](P1432) at ( 2.5,4.5){1432};
    		\node(P3241) at (-2,6){3241};
    		\node(P2431) at (-1,6){2431};
    		\node[blue](P3412) at ( 0,6){3412};
    		\node[blue](P4213) at ( 1,6){4213};
    		\node[blue](P4132) at ( 2,6){4132};
    		\node[blue](P3421) at (-1,7.5){3421};
    		\node(P4231) at ( 0,7.5){4231};
    		\node[blue](P4312) at ( 1,7.5){4312};
    		\node[blue](P4321) at (0,9){4321};
    		\draw[line width=0.5mm,blue](P1234) -- (P2134);
    		\draw[line width=0.5mm,red](P1234) -- (P1324);
    		\draw[line width=0.5mm,green](P1234) -- (P1243);
    		\draw(P2134) -- (P2314);
    		\draw[line width=0.5mm,green](P2134) -- (P2143);
    		\draw[line width=0.5mm,blue](P1324) -- (P3124);
    		\draw[line width=0.5mm,green](P1324) -- (P1342);
    		\draw(P1243) -- (P2143);
    		\draw[line width=0.5mm,red](P1243) -- (P1423);
    		\draw(P2314) -- (P3214);
    		\draw(P2314) -- (P2341);
    		\draw[line width=0.5mm,red](P3124) -- (P3214);
    		\draw[line width=0.5mm,green](P3124) -- (P3142);
    		\draw(P2143) -- (P2413);
    		\draw(P1342) -- (P3142);
    		\draw[line width=0.5mm,red](P1342) -- (P1432);
    		\draw[line width=0.5mm,blue](P1423) -- (P4123);
    		\draw(P1423) -- (P1432);
    		\draw(P3214) -- (P3241);
    		\draw(P2341) -- (P3241);
    		\draw(P2341) -- (P2431);
    		\draw[line width=0.5mm,red](P3142) -- (P3412);
    		\draw(P2413) -- (P4213);
    		\draw(P2413) -- (P2431);
    		\draw[line width=0.5mm,red](P4123) -- (P4213);
    		\draw(P4123) -- (P4132);
    		\draw[line width=0.5mm,blue](P1432) -- (P4132);
    		\draw(P3241) -- (P3421);
    		\draw(P2431) -- (P4231);
    		\draw[line width=0.5mm,blue](P3412) -- (P4312);
    		\draw[line width=0.5mm,green](P3412) -- (P3421);
    		\draw(P4213) -- (P4231);
    		\draw(P4132) -- (P4312);
    		\draw(P3421) -- (P4321);
    		\draw(P4231) -- (P4321);
    		\draw[line width=0.5mm,green](P4312) -- (P4321);	
    	\end{tikzpicture}
    	&
    	\begin{tikzpicture}[xscale=.9, yscale=0.7, color=lightgray]
    		\node[blue](P1234) at (0,0){1234};
    		\node[blue](P2134) at (-1,1.5){2134};
    		\node[blue](P1324) at ( 0,1.5){1324};
    		\node[blue](P1243) at ( 1,1.5){1243};
    		\node[blue](P2314) at (-2,3){2314};
    		\node(P3124) at (-1,3){3124};
    		\node[blue](P2143) at ( 0,3){2143};
    		\node[blue](P1342) at ( 1,3){1342};
    		\node[blue](P1423) at ( 2,3){1423};
    		\node[blue](P3214) at (-2.5,4.5){3214};
    		\node[blue](P2341) at (-1.5,4.5){2341};
    		\node(P3142) at (-0.5,4.5){3142};
    		\node[blue](P2413) at ( 0.5,4.5){2413};
    		\node(P4123) at ( 1.5,4.5){4123};
    		\node[blue](P1432) at ( 2.5,4.5){1432};
    		\node[blue](P3241) at (-2,6){3241};
    		\node[blue](P2431) at (-1,6){2431};
    		\node(P3412) at ( 0,6){3412};
    		\node[blue](P4213) at ( 1,6){4213};
    		\node(P4132) at ( 2,6){4132};
    		\node[blue](P3421) at (-1,7.5){3421};
    		\node[blue](P4231) at ( 0,7.5){4231};
    		\node(P4312) at ( 1,7.5){4312};
    		\node[blue](P4321) at (0,9){4321};
    		\draw[line width=0.5mm,blue](P1234) -- (P2134);
    		\draw[line width=0.5mm,red](P1234) -- (P1324);
    		\draw[line width=0.5mm,green](P1234) -- (P1243);
    		\draw[line width=0.5mm,red](P2134) -- (P2314);
    		\draw[line width=0.5mm,green](P2134) -- (P2143);
    		\draw(P1324) -- (P3124);
    		\draw[line width=0.5mm,green](P1324) -- (P1342);
    		\draw(P1243) -- (P2143);
    		\draw[line width=0.5mm,red](P1243) -- (P1423);
    		\draw[line width=0.5mm,blue](P2314) -- (P3214);
    		\draw[line width=0.5mm,green](P2314) -- (P2341);
    		\draw(P3124) -- (P3214);
    		\draw(P3124) -- (P3142);
    		\draw[line width=0.5mm,red](P2143) -- (P2413);
    		\draw(P1342) -- (P3142);
    		\draw[line width=0.5mm,red](P1342) -- (P1432);
    		\draw(P1423) -- (P4123);
    		\draw(P1423) -- (P1432);
    		\draw[line width=0.5mm,green](P3214) -- (P3241);
    		\draw(P2341) -- (P3241);
    		\draw[line width=0.5mm,red](P2341) -- (P2431);
    		\draw(P3142) -- (P3412);
    		\draw[line width=0.5mm,blue](P2413) -- (P4213);
    		\draw(P2413) -- (P2431);
    		\draw(P4123) -- (P4213);
    		\draw(P4123) -- (P4132);
    		\draw(P1432) -- (P4132);
    		\draw[line width=0.5mm,red](P3241) -- (P3421);
    		\draw[line width=0.5mm,blue](P2431) -- (P4231);
    		\draw(P3412) -- (P4312);
    		\draw(P3412) -- (P3421);
    		\draw(P4213) -- (P4231);
    		\draw(P4132) -- (P4312);
    		\draw[line width=0.5mm,blue](P3421) -- (P4321);
    		\draw(P4231) -- (P4321);
    		\draw(P4312) -- (P4321);
    	\end{tikzpicture}
		&
    	\begin{tikzpicture}[xscale=.9, yscale=0.7, color=lightgray]
			\node[blue](P1234) at (0,0){1234};
			\node[blue](P2134) at (-1,1.5){2134};
			\node[blue](P1324) at ( 0,1.5){1324};
			\node[blue](P1243) at ( 1,1.5){1243};
			\node[blue](P2314) at (-2,3){2314};
			\node[blue](P3124) at (-1,3){3124};
			\node[blue](P2143) at ( 0,3){2143};
			\node(P1342) at ( 1,3){1342};
			\node[blue](P1423) at ( 2,3){1423};
			\node[blue](P3214) at (-2.5,4.5){3214};
			\node(P2341) at (-1.5,4.5){2341};
			\node(P3142) at (-0.5,4.5){3142};
			\node[blue](P2413) at ( 0.5,4.5){2413};
			\node[blue](P4123) at ( 1.5,4.5){4123};
			\node[blue](P1432) at ( 2.5,4.5){1432};
			\node(P3241) at (-2,6){3241};
			\node[blue](P2431) at (-1,6){2431};
			\node(P3412) at ( 0,6){3412};
			\node[blue](P4213) at ( 1,6){4213};
			\node[blue](P4132) at ( 2,6){4132};
			\node(P3421) at (-1,7.5){3421};
			\node[blue](P4231) at ( 0,7.5){4231};
			\node[blue](P4312) at ( 1,7.5){4312};
			\node[blue](P4321) at (0,9){4321};
			\draw[line width=0.5mm,blue](P1234) -- (P2134);
			\draw[line width=0.5mm,red](P1234) -- (P1324);
			\draw[line width=0.5mm,green](P1234) -- (P1243);
			\draw[line width=0.5mm,red](P2134) -- (P2314);
			\draw[line width=0.5mm,green](P2134) -- (P2143);
			\draw[line width=0.5mm,blue](P1324) -- (P3124);
			\draw(P1324) -- (P1342);
			\draw(P1243) -- (P2143);
			\draw[line width=0.5mm,red](P1243) -- (P1423);
			\draw[line width=0.5mm,blue](P2314) -- (P3214);
			\draw(P2314) -- (P2341);
			\draw(P3124) -- (P3214);
			\draw(P3124) -- (P3142);
			\draw[line width=0.5mm,red](P2143) -- (P2413);
			\draw(P1342) -- (P3142);
			\draw(P1342) -- (P1432);
			\draw[line width=0.5mm,blue](P1423) -- (P4123);
			\draw[line width=0.5mm,green](P1423) -- (P1432);
			\draw(P3214) -- (P3241);
			\draw(P2341) -- (P3241);
			\draw(P2341) -- (P2431);
			\draw(P3142) -- (P3412);
			\draw[line width=0.5mm,blue](P2413) -- (P4213);
			\draw[line width=0.5mm,green](P2413) -- (P2431);
			\draw(P4123) -- (P4213);
			\draw[line width=0.5mm,green](P4123) -- (P4132);
			\draw(P1432) -- (P4132);
			\draw(P3241) -- (P3421);
			\draw(P2431) -- (P4231);
			\draw(P3412) -- (P4312);
			\draw(P3412) -- (P3421);
			\draw[line width=0.5mm,green](P4213) -- (P4231);
			\draw[line width=0.5mm,red](P4132) -- (P4312);
			\draw(P3421) -- (P4321);
			\draw[line width=0.5mm,red](P4231) -- (P4321);
			\draw(P4312) -- (P4321);
    	\end{tikzpicture}
    	&
    	\begin{tikzpicture}[xscale=.9, yscale=0.7, color=lightgray]
			\node[blue](P1234) at (0,0){1234};
			\node[blue](P2134) at (-1,1.5){2134};
			\node[blue](P1324) at ( 0,1.5){1324};
			\node[blue](P1243) at ( 1,1.5){1243};
			\node[blue](P2314) at (-2,3){2314};
			\node[blue](P3124) at (-1,3){3124};
			\node[blue](P2143) at ( 0,3){2143};
			\node[blue](P1342) at ( 1,3){1342};
			\node(P1423) at ( 2,3){1423};
			\node[blue](P3214) at (-2.5,4.5){3214};
			\node[blue](P2341) at (-1.5,4.5){2341};
			\node[blue](P3142) at (-0.5,4.5){3142};
			\node(P2413) at ( 0.5,4.5){2413};
			\node(P4123) at ( 1.5,4.5){4123};
			\node[blue](P1432) at ( 2.5,4.5){1432};
			\node[blue](P3241) at (-2,6){3241};
			\node[blue](P2431) at (-1,6){2431};
			\node[blue](P3412) at ( 0,6){3412};
			\node(P4213) at ( 1,6){4213};
			\node(P4132) at ( 2,6){4132};
			\node[blue](P3421) at (-1,7.5){3421};
			\node(P4231) at ( 0,7.5){4231};
			\node[blue](P4312) at ( 1,7.5){4312};
			\node[blue](P4321) at (0,9){4321};
			\draw[line width=0.5mm,blue](P1234) -- (P2134);
			\draw[line width=0.5mm,red](P1234) -- (P1324);
			\draw[line width=0.5mm,green](P1234) -- (P1243);
			\draw[line width=0.5mm,red](P2134) -- (P2314);
			\draw[line width=0.5mm,green](P2134) -- (P2143);
			\draw[line width=0.5mm,blue](P1324) -- (P3124);
			\draw[line width=0.5mm,green](P1324) -- (P1342);
			\draw(P1243) -- (P2143);
			\draw(P1243) -- (P1423);
			\draw[line width=0.5mm,blue](P2314) -- (P3214);
			\draw[line width=0.5mm,green](P2314) -- (P2341);
			\draw(P3124) -- (P3214);
			\draw[line width=0.5mm,green](P3124) -- (P3142);
			\draw(P2143) -- (P2413);
			\draw(P1342) -- (P3142);
			\draw[line width=0.5mm,red](P1342) -- (P1432);
			\draw(P1423) -- (P4123);
			\draw(P1423) -- (P1432);
			\draw[line width=0.5mm,green](P3214) -- (P3241);
			\draw(P2341) -- (P3241);
			\draw[line width=0.5mm,red](P2341) -- (P2431);
			\draw[line width=0.5mm,red](P3142) -- (P3412);
			\draw(P2413) -- (P4213);
			\draw(P2413) -- (P2431);
			\draw(P4123) -- (P4213);
			\draw(P4123) -- (P4132);
			\draw(P1432) -- (P4132);
			\draw[line width=0.5mm,red](P3241) -- (P3421);
			\draw(P2431) -- (P4231);
			\draw[line width=0.5mm,blue](P3412) -- (P4312);
			\draw(P3412) -- (P3421);
			\draw(P4213) -- (P4231);
			\draw(P4132) -- (P4312);
			\draw[line width=0.5mm,blue](P3421) -- (P4321);
			\draw(P4231) -- (P4321);
			\draw(P4312) -- (P4321);
    	\end{tikzpicture}
	\end{tabular}
}
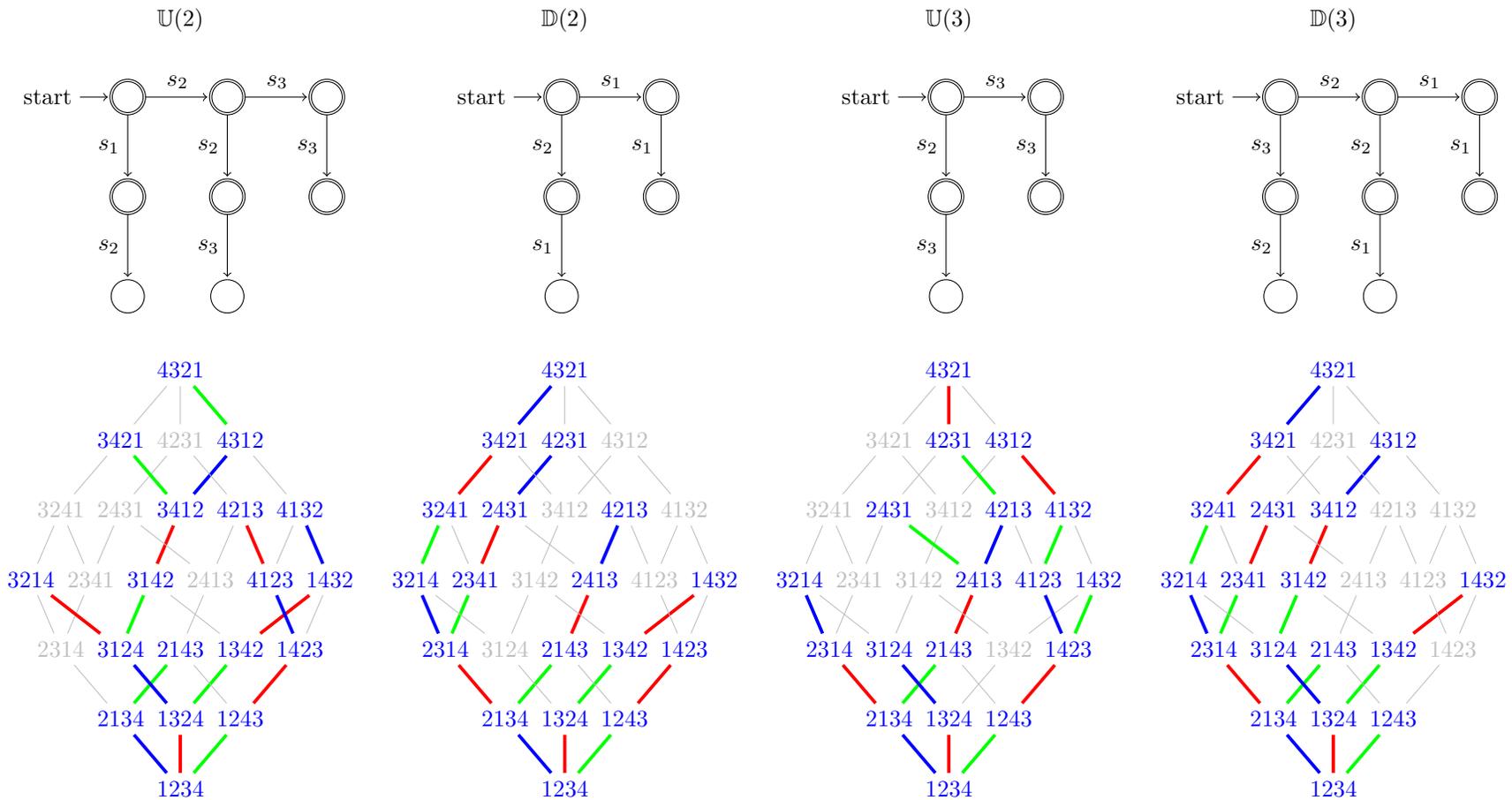
{Generating trees for the $(\{j\}, \varnothing)$- and $(\varnothing, \{j\})$-permutree minimal permutations of~$\fS_4$, with priority order $s_1 \prec s_2 \prec s_3$.}
{fig:TreePartialOrientations}


\section{Intersection of automata}\label{sec:intersectionsAutomata}

We now consider arbitrary subsets~$U$ and~$D$ of~$\{2, \dots, n-1\}$.
We already know from~\cite{PilaudPons-permutrees} and \cref{thm:patternAvoidance} that the following conditions are equivalent for~$\pi \in \fS_n$:
\begin{enumerate}[(i)]
	\item the permutation $\pi$ is minimal in its~$(U,D)$-permutree class, 
	\item for~$i < j < k$, the permutation $\pi$ does not contain the subword~$jki$ if~$j \in U$ and~$kij$ if~$j \in D$,
	\item for each~$j \in U$ (each~$j \in D$), there is a reduced word for~$\pi$ accepted by~$\automatonU(j)$ (resp.~by~$\automatonD(j)$).
\end{enumerate}
A natural question is whether there is a reduced word simultaneously accepted by all these automata.
We start with an example showing that this is not always the case.

\begin{example}\label{exm:problemUDintersect}
For~$j \in \{2, \dots, n-1\}$, consider $U = \{j\} = D$, and~$\pi = s_{j-1} \cdot s_j \cdot s_{j-1} = s_j \cdot s_{j-1} \cdot s_j$.
Then, the word~$s_{j-1} \cdot s_j \cdot s_{j-1}$ is accepted by $\automatonD(j)$ but not by~$\automatonU(j)$, while the word $s_j \cdot s_{j-1} \cdot s_j$ is accepted by $\automatonU(j)$ but not by~$\automatonD(j)$.
\end{example}

This example clearly extends to all subsets~$U$ and~$D$ of~$\{2, \dots, n-1\}$ with a non-empty intersection.
In contrast, we will now show that this situation cannot occur when~$U$ and~$D$ are disjoint.

\subsection{Proof of \cref{thm:permutreeMinimal}}

\cref{exm:problemUDintersect} motivates \cref{thm:permutreeMinimal} that we repeat here for convenience.

\newtheorem*{thm:permutreeMinimal}{\cref{thm:permutreeMinimal}}
\begin{thm:permutreeMinimal}
Consider two disjoint subsets~$U$ and~$D$ of~$\{2, \dots, n-1\}$.
The following conditions are equivalent for~$\pi \in \fS_n$:
\begin{itemize}
\item $\pi$ admits a reduced word accepted by all automata~$\automatonU(j)$ for~$j \in U$ and~$\automatonD(j)$ for~${j \in D}$,
\item $\pi$ contains no subword $jki$ if~$j \in U$ and~$kij$ if~$j \in D$ for any~$i < j < k$.
\end{itemize}
\end{thm:permutreeMinimal}

\begin{proof}
The direct implication is immediate from \cref{thm:patternAvoidance}.
For the converse implication, consider a permutation~$\pi$ that avoids~$jki$ for~$j \in U$ and~$kij$ for~$j \in D$.
Consider $\bar U \eqdef \set{j \in U}{\ninv_j(\pi) \ne 0}$ and~$\bar D \eqdef \set{j \in D}{\ninv^j(\pi) \ne 0}$.
By \cref{prop:sameStateAcceptedReducedExpressionsRefined}\,(ii), any reduced word for~$\pi$ is accepted by~$\automatonU(j)$ for~$j \in U \ssm \bar U$ and by~$\automatonD(j)$ for~$j \in D \ssm \bar D$.
We can therefore assume that~$\bar U = U$ and~$\bar D = D$ and one of them is non-empty, say~$\bar U = U \ne \varnothing$.
Let~$j_\circ \eqdef \max(U)$ and~$m$ be minimal such that~$j_\circ < m$ and~$\pi^{-1}(j_\circ) > \pi^{-1}(m)$.
By minimality of~$m$, we obtain that~$\pi$ contains the subword~$mj_\circ\ell$ for any~$j_\circ < \ell < m$.
It implies that
\begin{itemize}
\item $\ell$ is neither in~$U$ by maximality of~$j_\circ$, nor in~$D$ by assumption on~$\pi$, for all~$j_\circ < \ell < m$,
\item $\pi$ reverses~$\ell$ and $m$ for all $\ell$ in $\{j_\circ,\dots,m-1\}$, so that $\pi$ admits a reduced word of the form $\pi = s_{m-1} \cdots s_{j_\circ} \cdot \tau$.
\end{itemize}
Lemmas \ref{lem:vincent2}\,(2) and \ref{lem:vincent4}\,(2) ensure that
\begin{itemize}
\item $\tau$ avoids~$jki$ for all~$j \in U \ssm \{j_\circ\}$ and~$kij$ for all~$j \in D \ssm \{m\}$ (because~$j_\circ, \dots, m-1$ are all distinct from~$j-1$ and~$j$ in these cases by the second observation above),
\item $\tau$ avoids $(j_\circ+1)ki$,
\item if~$m \in D$, then~$\tau$ avoids~$kij_\circ$.
\end{itemize}
By induction, it follows that~$\tau$ admits a reduced word~$w$ simultaneously accepted by all automata~$\automatonU(j)$ for~$j \in U \ssm \{j_\circ\}$ and~$j = j_\circ + 1$, and all automata~$\automatonD(j)$ for~$j \in D \ssm \{m\}$ and~$j = j_\circ$ if~$m \in D$.
By Lemmas \ref{lem:vincent2}\,(1) and \ref{lem:vincent4}\,(1), we conclude that $s_{m-1} \cdots s_{j_\circ} \cdot w$ is a reduced word for~$\pi$ simultaneously accepted by all $\automatonU(j)$ for~$j \in U$ and~$\automatonD(j)$ for~$j \in D$.
\end{proof}

\subsection{Intersection of automata}

\cref{thm:permutreeMinimal} can be rephrased in terms of intersection of automata.
Recall that the intersection of some automata~$\automatonA_1, \dots, \automatonA_p$ is the automaton~$\automatonA = \bigcap_{i \in [p]} \automatonA_i$ such that a word is accepted by~$\automatonA$ if and only if it is accepted by all~$\automatonA_1, \dots, \automatonA_p$.
A state of the automaton~$\automatonA$ is $p$-tuple formed by states of the automata~$\automatonA_1, \dots, \automatonA_p$, and a transition~$t$ simultaneously changes all entries of the $p$-tuple corresponding to states modified by~$t$.
See~\cite[p.\,59--60]{HopcroftUllman} for details.
We denote by~$\automatonP(U,D)$ the intersection of the automata~$\automatonU(j)$ for~$j \in U$ and~$\automatonD(j)$ for~$j \in D$.
We thus obtain the following statement.

\begin{corollary}
\label{coro:permutreeMinimal}
When~$U$ and~$D$ are disjoint, the following conditions are equivalent for~$\pi \in \fS_n$:
\begin{itemize}
	\item $\pi$ admits a reduced word accepted by the automaton~$\automatonP(U,D)$,
	\item $\pi$ contains no subword~$jki$ if~$j \in U$ and~$kij$ if~$j \in D$ with~$i < j < k$.
\end{itemize}
\end{corollary}

We say that a state of~$\automatonP(U,D)$ is \defn{healthy} (resp.~\defn{ill}, resp.~\defn{dead}) when the corresponding states in~$\automatonU(j)$ for~$j \in U$ and~$\automatonD(j)$ for~$j \in D$ are all healthy (resp.~contain at least one ill state, but no dead one, resp.~contains at least one dead state).
\cref{fig:automataProduct} illustrates the automata~$\automatonP(\{4\},\{2\})$ when~$n = 5$ (left), $\automatonP(\{3\},\{2\})$ for $n=4$ (middle), and~$\automatonP(\{2\},\{4\})$ for $n=5$ (right).
For the first two automata, we draw the complete automata on top, and their skeleta on the bottom.
Here, we call skeleton a simplification of the automaton that recognizes the same reduced words.
It is obtained using the fact that the word is rejected as soon as we reach a dead state, and that the automata~$\automatonU(n)$ and~$\automatonD(1)$ accept all reduced words.
For the last automaton, the complete intersection is too big, so we only draw the reachable healthy states.
We color the transitions in red, blue, or purple depending on whether \mbox{only~$\automatonU$, only $\automatonD$, or both~$\automatonU$ and~$\automatonD$ change state}.
\hvFloat[floatPos=p, capWidth=h, capPos=r, capAngle=90, objectAngle=90, capVPos=c, objectPos=c]{figure}
{
\begin{tabular}{l@{\hspace{-.8cm}}l@{\hspace{-.8cm}}l}
	\begin{tikzpicture}[shorten >=1pt, node distance=2cm, on grid, auto, baseline=-1.5cm]
		\node[state,initial,accepting,minimum size=0.5cm] (11) {}; 
		\node[state,accepting,,minimum size=0.5cm] (12) [right= 1.5cm of 11] {};
		\node[state,accepting,,minimum size=0.5cm] (21) [right= 2.5cm of 12] {};
		\node[state,accepting,,minimum size=0.5cm] (22) [right= 1.5cm of 21] {};
		\node[state,accepting,minimum size=0.5cm] (13) [below= 1cm of 11] {};
		\node[state,accepting,minimum size=0.5cm] (14) [right= 1.5cm of 13] {};
		\node[state,accepting,minimum size=0.5cm] (23) [right= 2.5cm of 14] {};
		\node[state,accepting,minimum size=0.5cm] (24) [right= 1.5cm of 23] {};
		\node[state,accepting,minimum size=0.5cm] (31) [below= 1.5cm of 14] {};
		\node[state,accepting,minimum size=0.5cm] (32) [right= 1.5cm of 31] {};
		\node[state,minimum size=0.5cm] (15) [below= 1cm of 13] {};
		\node[state,minimum size=0.5cm] (25) [below= 1cm of 23] {};
		\node[state,accepting,minimum size=0.5cm] (33) [below= 1cm of 31] {};
		\node[state,accepting,minimum size=0.5cm] (34) [right= 1.5cm of 33] {};
		\node[state,accepting,minimum size=0.5cm] (41) [right= 2.5cm of 32] {};
		\node[state,accepting,minimum size=0.5cm] (42) [right= 1.5cm of 41] {};
		\node[state,minimum size=0.5cm] (35) [below= 1cm of 33] {};
		\node[state,accepting,minimum size=0.5cm] (43) [below= 1cm of 41] {};
		\node[state,accepting,minimum size=0.5cm] (44) [right= 1.5cm of 43] {};
		\node[state,minimum size=0.5cm] (51) [below= 1.5cm of 34] {};
		\node[state,minimum size=0.5cm] (52) [right= 1.5cm of 51] {};
		\node[state,minimum size=0.5cm] (45) [below= 1cm of 43] {};
		\node[state,minimum size=0.5cm] (53) [below= 1cm of 51] {};
		\node[state,minimum size=0.5cm] (54) [right= 1.5cm of 53] {};
		\node[state,minimum size=0.5cm] (55) [below= 1cm of 53] {};
		\path[->] (11) edge [bend left,color=blue] node {$s_1$} (21) edge [color=blue] node [below=1mm] {$s_2$} (31) edge [color=red] node [swap] {$s_3$} (13) edge [color=red] node [swap] {$s_4$} (12);
		\path[->] (12) edge [bend left,color=blue] node {$s_1$} (22) edge [color=blue] node [below=1mm] {$s_2$} (32) edge [color=red] node [swap] {$s_4$} (14);
		\path[->] (13) edge [bend left,color=blue] node {$s_1$} (23) edge [color=blue] node [below=1mm] {$s_2$} (33) edge [color=red] node [swap] {$s_4$} (15);
		\path[->] (14) edge [bend left,color=blue] node {$s_1$} (24) edge [color=blue] node [left=-1mm] {$s_2$} (34);
		\path[->] (15) edge [color=blue] node [below=1mm] {$s_2$} (35);
		\path[->] (21) edge [color=blue] node [below=1mm] {$s_1$} (41) edge [color=red] node [swap] {$s_3$} (23) edge [color=red] node [swap] {$s_4$} (22);
		\path[->] (22) edge [color=blue] node [below=1mm] {$s_1$} (42) edge [color=red] node [swap] {$s_4$} (24);
		\path[->] (23) edge [color=blue] node [below=1mm] {$s_1$} (43) edge [color=red] node [swap] {$s_4$} (25);
		\path[->] (24) edge [color=blue] node [left=-1mm] {$s_1$} (44);
		\path[->] (25) edge [color=blue] node [below=1mm] {$s_1$} (45);
		\path[->] (31) edge [color=blue] node [below=1mm] {$s_1$} (51) edge [color=red] node {$s_3$} (33) edge [color=red] node [below left=-1mm] {$s_4$} (32);
		\path[->] (32) edge [color=blue] node [below=1mm] {$s_1$} (52) edge [color=red] node {$s_4$} (34);
		\path[->] (33) edge [color=blue] node [below=1mm] {$s_1$} (53) edge [color=red] node {$s_4$} (35);
		\path[->] (34) edge [color=blue] node [left=-1mm] {$s_1$} (54);
		\path[->] (35) edge [color=blue] node [below=1mm] {$s_1$} (55);
		\path[->] (41) edge [color=red] node {$s_3$} (43) edge [color=red] node [below left=-1mm] {$s_4$} (42);
		\path[->] (42) edge [color=red] node {$s_4$} (44);
		\path[->] (43) edge [color=red] node {$s_4$} (45);
		\path[->] (51) edge [color=red] node {$s_3$} (53) edge [color=red] node [below left=-1mm] {$s_4$} (52);
		\path[->] (52) edge [color=red] node {$s_4$} (54);
		\path[->] (53) edge [color=red] node {$s_4$} (55);
	\end{tikzpicture}
	&
	\begin{tikzpicture}[shorten >=1pt, node distance=2cm, on grid, auto, baseline=-1.5cm]
		\node[state,initial,accepting,minimum size=0.5cm] (11) {}; 
		\node[state,accepting,,minimum size=0.5cm] (12) [right= 1.5cm of 11] {};
		\node[state,accepting,minimum size=0.5cm] (14) [below= 1cm of 12] {};
		\node[state,accepting,minimum size=0.5cm] (21) [right= 2.5cm of 12] {};
		\node[state,accepting,minimum size=0.5cm] (22) [right= 1.5cm of 21] {};		
		\node[state,accepting,minimum size=0.5cm] (23) [right= 2.5cm of 14] {};
		\node[state,accepting,minimum size=0.5cm] (24) [right= 1.5cm of 23] {};
		\node[state,minimum size=0.5cm] (25) [below= 1cm of 23] {};
		\node[state,accepting,minimum size=0.5cm] (33) [below= 2.5cm of 14] {};
		\node[state,accepting,minimum size=0.5cm] (34) [right= 1.5cm of 33] {};
		\node[state,accepting,minimum size=0.5cm] (32) [above= 1cm of 34] {};
		\node[state,minimum size=0.5cm] (35) [below= 1cm of 33] {};
		\node[state,accepting,minimum size=0.5cm] (41) [below= 1.5cm of 24] {};
		\node[state,accepting,minimum size=0.5cm] (42) [right= 1.5cm of 41] {};
		\node[state,accepting,minimum size=0.5cm] (43) [below= 1cm of 41] {};
		\node[state,accepting,minimum size=0.5cm] (44) [right= 1.5cm of 43] {};
		\node[state,minimum size=0.5cm] (45) [below= 1cm of 43] {};
		\node[state,minimum size=0.5cm] (53) [below= 2.5cm of 34] {};
		\node[state,minimum size=0.5cm] (55) [below= 1cm of 53] {};
		\node[state,minimum size=0.5cm] (54) [right= 1.5cm of 53] {};
		\node[state,minimum size=0.5cm] (52) [above= 1cm of 54] {};
		\path[->] (11) edge [bend left,color=blue] node {$s_1$} (21) edge [color=violet] node [below=2mm] {$s_2$} (33) edge [color=red] node [swap] {$s_3$} (12);
		\path[->] (12) edge [bend left,color=blue] node {$s_1$} (22) edge [color=blue] node [below=1mm] {$s_2$} (32) edge [color=red] node [swap] {$s_3$} (14);
		\path[->] (14) edge [bend left,color=blue] node {$s_1$} (24) edge [color=blue] node [below=1mm] {$s_2$} (34);
		\path[->] (21) edge [color=blue] node [below=1mm] {$s_1$} (41) edge [color=red] node [swap] {$s_2$} (23) edge [color=red] node [swap] {$s_3$} (22);
		\path[->] (22) edge [color=blue] node [below=1mm] {$s_1$} (42) edge [color=red] node [swap] {$s_3$} (24);
		\path[->] (23) edge [color=blue] node [below=1mm] {$s_1$} (43) edge [color=red] node [swap] {$s_3$} (25);
		\path[->] (24) edge [color=blue] node [left=-1mm] {$s_1$} (44);
		\path[->] (25) edge [color=blue] node [below=1mm] {$s_1$} (45);
		\path[->] (32) edge [color=blue] node [below=1mm] {$s_1$} (52) edge [color=red] node {$s_3$} (34);
		\path[->] (33) edge [color=red] node [swap] {$s_3$} (35) edge [color=blue] node [below=1mm] {$s_1$} (53);
		\path[->] (34) edge [color=blue] node [below=1mm] {$s_1$} (54);
		\path[->] (35) edge [color=blue] node [below=1mm] {$s_1$} (55);
		\path[->] (41) edge [color=red] node {$s_2$} (43) edge [color=red] node [below left=-1mm] {$s_3$} (42);
		\path[->] (42) edge [color=red] node {$s_3$} (44);
		\path[->] (43) edge [color=red] node {$s_3$} (45);
		\path[->] (52) edge [color=red] node {$s_3$} (54);
		\path[->] (53) edge [color=red] node {$s_3$} (55);
	\end{tikzpicture}
	&
	\begin{tikzpicture}[shorten >=1pt, node distance=2cm, on grid, auto, baseline=-1.5cm]
		\node[state,initial,accepting,minimum size=0.5cm] (11) {}; 
		\node[state,accepting,,minimum size=0.5cm] (12) [right= 1.5cm of 11] {};
		\node[state,accepting,,minimum size=0.5cm] (23) [right= 1.5cm of 12] {};
		\node[state,accepting,,minimum size=0.5cm] (24) [right= 1.5cm of 23] {};
		\node[state,accepting,minimum size=0.5cm] (21) [below= 1.5cm of 11] {};
		\node[state,accepting,minimum size=0.5cm] (32) [below= 1.5cm of 12] {};
		\node[state,accepting,minimum size=0.5cm] (33) [below= 1.5cm of 23] {};
		\node[state,accepting,minimum size=0.5cm] (34) [below= 1.5cm of 24] {};
		\node[state,accepting,minimum size=0.5cm] (42) [below= 1.5cm of 32] {};
		\node[state,accepting,minimum size=0.5cm] (43) [below= 1.5cm of 33] {};
		\node[state,accepting,minimum size=0.5cm] (44) [below= 1.5cm of 34] {};
		\path[->] (11) edge [color=red] node [swap] {$s_2$} (21) edge [color=blue] node {$s_3$} (12);
		\path[->] (12) edge [color=violet] node {$s_2$} (23);
		\path[->] (23) edge [color=red] node [swap] {$s_3$} (33) edge [color=blue] node {$s_1$} (24);
		\path[->] (24) edge [color=red] node [swap] {$s_3$} (34);
		\path[->] (21) edge [color=violet] node {$s_3$} (32);
		\path[->] (32) edge [color=blue] node {$s_2$} (33) edge [color=red] node [swap] {$s_4$} (42);
		\path[->] (33) edge [color=blue] node {$s_1$} (34) edge [color=red] node [swap] {$s_4$} (43);
		\path[->] (34) edge [color=red] node [swap] {$s_4$} (44);
		\path[->] (42) edge [color=blue] node {$s_2$} (43);
		\path[->] (43) edge [color=blue] node {$s_1$} (44);
	\end{tikzpicture}
	\\
	\\[.3cm]
	\begin{tikzpicture}[shorten >=1pt, node distance=2cm, on grid, auto, baseline=-1.5cm]
		\node[state,initial,accepting,minimum size=0.5cm] (h1) {}; 
		\node[state,accepting,,minimum size=0.5cm] (h2) [right= 1.5cm of h1] {};
		\node[state,accepting,minimum size=0.5cm] (h3) [right= 2.5cm of h2] {};
		\node[state,accepting,minimum size=0.5cm] (h4) [right= 1.5cm of h3] {};
		\node[state,accepting,minimum size=0.5cm] (i1) [below= 1cm of h1] {};
		\node[state,accepting,minimum size=0.5cm] (i11) [right= 1cm of i1] {};
		\node[state,accepting,minimum size=0.5cm] (i12) [below= 1cm of i11] {};
		\node[state,accepting,minimum size=0.5cm] (i2) [right= 1.5cm of i11] {};
		\node[state,minimum size=0.5cm] (d) [below= 2.5cm of i2] {};
		\node[state,accepting,minimum size=0.5cm] (i3) [below= 1.1cm of h3] {};
		\path[->] (h1) edge [bend left,color=blue] node {$s_1$} (h3) edge [swap,color=red] node {$s_3$} (i1) edge [color=red] node [swap] {$s_4$} (h2) edge [color=blue] node [below left=-2mm] {$s_2$} (i11) ;
		\path[->] (i1) edge [bend right=50,swap,color=red] node {$s_4$} (d) edge [bend left, color=blue] node {$s_1$} (i3) edge [color=blue] node [below left=-2mm] {$s_2$} (i12);
		\path[->] (i11) edge [swap,color=red] node {$s_3$} (i12) edge [color=blue] node {$s_1$} (d) edge [color=red] node [swap] {$s_4$} (i2);
		\path[->] (i12) edge [color=violet] node [below left=-2mm] {${\color{blue}s_1},{\color{red}s_4}$} (d);
		\path[->] (h2) edge [bend left,color=blue] node {$s_1$} (h4) edge [color=blue] node [below left=-2mm] {$s_2$} (i2);
		\path[->] (i2) edge [color=blue] node {$s_1$} (d);
		\path[->] (h3) edge [color=red] node [swap] {$s_4$} (h4) edge [swap,color=red] node {$s_3$} (i3);
		\path[->] (i3) edge [color=red] node {$s_4$} (d);
	\end{tikzpicture}
	&
	\begin{tikzpicture}[shorten >=1pt, node distance=2cm, on grid, auto, baseline=-1.5cm]
		\node[state,initial,accepting,minimum size=0.5cm] (h1) {}; 
		\node[state,accepting,,minimum size=0.5cm] (h2) [right= 1.5cm of h1] {};
		\node[state,accepting,minimum size=0.5cm] (h3) [right= 2.5cm of h2] {};
		\node[state,accepting,minimum size=0.5cm] (h4) [right= 1.5cm of h3] {};
		\node[state,accepting,minimum size=0.5cm] (i1) [below= 1cm of h1] {};
		\node[state,accepting,minimum size=0.5cm] (i2) [below= 1cm of h2] {};
		\node[state,minimum size=0.5cm] (d) [below= 1cm of i2] {};
		\node[state,accepting,minimum size=0.5cm] (i3) [below= 1cm of h3] {};
		\path[->] (h1) edge [bend left,color=blue] node {$s_1$} (h3) edge [color=violet] node [swap] {$s_2$} (i1) edge [color=red] node [swap] {$s_3$} (h2);
		\path[->] (i1) edge [color=violet] node [below left=-2mm] {${\color{blue}s_1},{\color{red}s_3}$} (d);
		\path[->] (h2) edge [bend left,color=blue] node {$s_1$} (h4) edge [color=blue] node {$s_2$} (i2);
		\path[->] (i2) edge [color=blue] node {$s_1$} (d);
		\path[->] (h3) edge [color=red] node [swap] {$s_3$} (h4) edge [swap,color=red] node {$s_2$} (i3);
		\path[->] (i3) edge [color=red] node {$s_3$} (d);
	\end{tikzpicture}
\end{tabular}
}
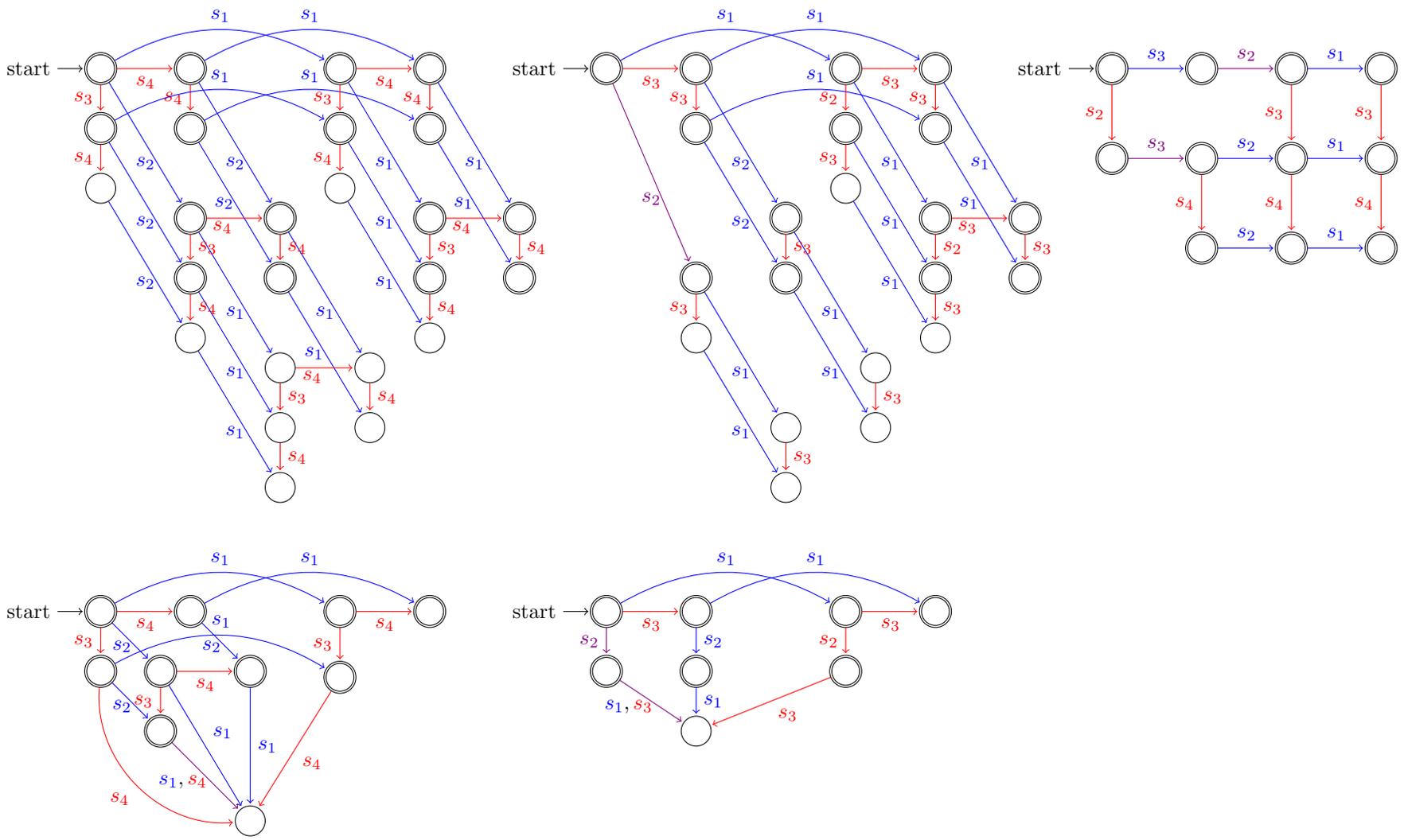
{The automaton~$\automatonP(\{4\},\{2\})$ for $n = 5$ and its skeleton (left), the automaton~$\automatonP(\{3\},\{2\})$ for $n=4$ and its skeleton (middle), and the healthy states of the automaton~$\automatonP(\{2\},\{4\})$ for $n=5$ (right).}
{fig:automataProduct}

\subsection{The set of accepted reduced words of~$\automatonP(U,D)$}

Applying the principles of \cref{subsec:principles} to each automaton~$\automatonU(j)$ for~$j \in U$ and~$\automatonD(j)$ for~$j \in D$, we derive similar principles for the automaton~$\automatonP(U,D)$.
The following statements are direct consequences of \cref{prop:prefixesReducedExpressions,prop:algorithm,prop:sameStateAcceptedReducedExpressions}.

\begin{proposition}\label{prop:prefixesReducedExpressionsIntersection}
The set of reduced words accepted by~$\automatonP(U,D)$ is closed by prefix.
\end{proposition}

\begin{proposition}\label{prop:algorithmIntersection}
If a permutation~$\pi$ avoids~$jki$ for~$j \in U$ and~$kij$ for~$j \in D$, and admits a reduced word starting with~$s_\ell$ such that the transition~$s_\ell$ leads to a healthy state of~$\automatonP(U,D)$, then it admits a reduced word starting with~$s_\ell$ and accepted by~$\automatonP(U,D)$.
\end{proposition}

\begin{proposition}\label{prop:sameStateAcceptedReducedExpressionsIntersection}
Given a permutation $\pi \in \fS_n$, all the reduced words for~$\pi$ accepted by $\automatonP(U,D)$ end at the same state.
\end{proposition}

\subsection{Permutree sorting}\label{subsec:permutreeSorting}

We have seen that the set of reduced words for any~$(U,D)$-permutree minimal permutation~$\pi$ that are accepted by~$\automatonP(U,D)$ is non-empty (by \cref{coro:permutreeMinimal}) and that the set of all reduced words that are accepted by~$\automatonP(U,D)$ is closed by prefix (by \cref{prop:prefixesReducedExpressionsIntersection}).
Therefore, it is possible to sort~$\pi$ passing only through $(U,D)$-permutree minimal permutations along the way.
This motivates the following definition.

\begin{definition}
An \defn{$(U,D)$-permutree sorting algorithm} is a sorting procedure such that
\begin{itemize}
\item applied to a $(U,D)$-permutree minimal permutation~$\pi$, it only passes through $(U,D)$-permutree minimal permutations and arrives to the identity permutation, 
\item it fails to sort a non $(U,D)$-permutree minimal permutation~$\pi$.
\end{itemize}
\end{definition}

\begin{example}
The stack sorting algorithm is a $(\{2, \dots, n-1\}, \varnothing)$-permutree sorting algorithm.
\end{example}

Said differently, any procedure that looks for a reduced word accepted by~$\automatonP(U,D)$ gives a $(U,D)$-permutree sorting algorithm.
For instance, \cref{algo:shortcutsUj} is a $(\{j\}, \varnothing)$-permutree sorting algorithm.
We generalize it in the following algorithm, where we opted for a recursive style.
As in \cref{algo:shortcutsUj}, the algorithm will read the automaton~$\automatonP(U,D)$ without actually constructing it.
To virtually follow the edges of the automaton~$\automatonP(U,D)$, we use two operations on our sets~$U$~and~$D$:
\[
\moveU(U,\ell) = 
\begin{cases}
U & \text{ if } \ell \notin U, \\
(U \ssm \{ \ell \}) \cup \{ \ell + 1 \} & \text{ if } \ell \in U,
\end{cases}
\]

\[
\moveD(D,\ell) = 
\begin{cases}
D & \text{ if } \ell + 1 \notin D, \\
(D \ssm \{ \ell + 1 \}) \cup \{ \ell \} & \text{ if } \ell + 1 \in D.
\end{cases}
\]

\bigskip
\IncMargin{1em}
\SetKwInOut{Input}{Input}\SetKwInOut{Output}{Output}
\SetKwFor{Repeat}{repeat}{}{}
\SetKwIF{If}{ElseIf}{Else}{if}{then}{else if}{else}{}
\SetKwProg{Fn}{Function}{}{}
\DontPrintSemicolon
\begin{algorithm}[H]
	\Fn{{\rm permutreeSort$(\pi, U, D)$}}{
	\Input{a permutation $\pi \in \fS_n$ and two disjoint subsets~$U$ and~$D$ of~$[n]$}
	\Output{a reduced word accepted by~$\automatonP(U,D)$, candidate reduced word for~$\pi$}
	\If{$\exists \; \ell \in [n-1]$ such that $\ell$ and $\ell+1$ are reversed in~$\pi$, and $\ell+1 \notin U$ and~$\ell \notin D$}{
		\Return $s_\ell \cdot \text{permutreeSort}(s_\ell \cdot \pi, \; \moveU(U, \ell), \; \moveD(D,\ell))$ \;
	}
	\If{$\exists \; \ell \in [n-1]$ such that $\ell$ and $\ell+1$ are reversed in~$\pi$,
	\\ and ($\ell + 1 \notin U$ or $\pi([\ell+1]) = [\ell +1]$) and ($\ell \notin D$ or $\pi([\ell -1]) = [\ell - 1]$)
	}{
		\Return $s_\ell \cdot \text{permutreeSort}(s_\ell \cdot \pi, \; \moveU(U \ssm \{ \ell + 1 \}, \ell), \; \moveD(D \ssm \{ \ell \},\ell))$ \;
	}
	\Return $\varepsilon$ \;
	}
	\caption{$(U,D)$-permutree sorting}
	\label{algo:permutreeSorting}
\end{algorithm}
\bigskip

Note that in \cref{algo:permutreeSorting}, we could ignore~$n$ in the list~$U$ (resp.~$1$ in the list~$D$) since~$\automatonU(n)$ (resp.~$\automatonD(1)$) accepts all reduced words.
We have decided not to do it to be coherent with our recursive definition of~$\automatonU(j)$ and~$\automatonD(j)$.

\begin{example}\label{exm:algo2}
	We present in \cref{tab:algo2-3-2} the $(\{3\},\{2\})$-permutree sorting algorithm in action for the permutations~$\pi_1 \eqdef 3214$, $\pi_2 \eqdef 1324$ and~$\pi_3 \eqdef 1342$, and in \cref{tab:algo2-2-4} the $(\{2\},\{4 \})$-permutree sorting algorithm in action for the permutations~$\pi_4 \eqdef 54213$ and~$\pi_5 \eqdef 15342$. 
	The corresponding automata~$\automatonP(\{3\},\{2\})$ and~$\automatonP(\{2\},\{4\})$ are represented in \cref{fig:automataProduct}.
	Each row in these tables contains the states of the permutation~$\pi$ and of the word~$w$, the current values of~$U$, $D$ and~$\ell$ in use at each step, and the values of~$k$ for which we have to check that~$\pi([k]) = [k]$, crossed in red when it fails. These tables show that~$\pi_1$ and~$\pi_2$ are $(\{3\},\{2\})$-permutree sortable while $\pi_3$ is not, and that~$\pi_4$ is $(\{2\},\{4 \})$-permutree sortable while $\pi_5$ is not.

\begin{table}[h!]
	\centerline{
	\begin{tabular}[t]{l|l|c|c|c|c}
		$\pi_1$ & $w_1$ & $U_1$ & $D_1$ & $\ell_1$ & $k_1$ \\
		\hline
		$3214$ & $\varepsilon$ & $\{3\}$ & $\{2\}$ & $1$ & . \\
		$3124$ & $s_1$ & $\{3\}$ & $\{1\}$ & $2$ & $3$ \\
		$2134$ & $s_1 \cdot s_2$ & $\varnothing$ & $\{1\}$ & $1$ & $0$ \\
		$1234$ & $s_1 \cdot s_2 \cdot s_1$ & & &
	\end{tabular}
	\quad
	\begin{tabular}[t]{l|l|c|c|c|c}
		$\pi_2$ & $w_2$ & $U_2$ & $D_2$ & $\ell_2$ & $k_2$ \\
		\hline
		$1324$ & $\varepsilon$ & $\{3\}$ & $\{2\}$ & $2$ & $1$, $3$ \\
		$1234$ & $s_2$ & & &
	\end{tabular}
	\quad
	\begin{tabular}[t]{l|l|c|c|c|c} 
		$\pi_3$ & $w_3$ & $U_3$ & $D_3$ & $\ell_3$ & $k_3$ \\
		\hline
		$1342$ & $\varepsilon$ & $\{3\}$ & $\{2\}$ & $2$ & $1$, \textcolor{red}{$\xcancel{\textcolor{black}{3}}$}
	\end{tabular}
	}
	\caption{$(\{3\},\{2 \})$-permutree sorting of~$\pi_1 \eqdef 3214$, $\pi_2 \eqdef 1324$ and~$\pi_3 \eqdef 1342$.}
	\label{tab:algo2-3-2}
\end{table}

\begin{table}[h!]
	\centerline{
	\begin{tabular}[t]{l|l|c|c|c|c}
		$\pi_4$ & $w_4$ & $U_4$ & $D_4$ & $\ell_4$ & $k_4$ \\
		\hline
		$54213$ & $\varepsilon$ & $\{2\}$ & $\{4\}$ & $3$ & . \\
		$53214$ & $s_3$ & $\{2\}$ & $\{3\}$ & $2$ & . \\
		$52314$ & $s_3 \cdot s_2$ & $\{3\}$ & $\{2\}$ & $1$ & . \\
		$51324$ & $s_3 \cdot s_2 \cdot s_1$ & $\{3\}$ & $\{1\}$ & $4$ & . \\
		$41325$ & $s_3 \cdot s_2 \cdot s_1 \cdot s_4$ & $\{3\}$ & $\{1\}$ & $3$ & . \\
		$31425$ & $s_3 \cdot s_2 \cdot s_1 \cdot s_4 \cdot s_3$ & $\{4\}$ & $\{1\}$ & $2$ & . \\
		$21435$ & $s_3 \cdot s_2 \cdot s_1 \cdot s_4 \cdot s_3 \cdot s_2$ & $\{4\}$ & $\{1\}$ & $1$ & . \\
		$12435$ & $s_3 \cdot s_2 \cdot s_1 \cdot s_4 \cdot s_3 \cdot s_2 \cdot s_1$ & $\{4\}$ & $\{1\}$ & $3$ & $4$ \\
		$12345$ & $s_3 \cdot s_2 \cdot s_1 \cdot s_4 \cdot s_3 \cdot s_2 \cdot s_1 \cdot s_3$ & $\{4\}$ & $\{1\}$ &
	\end{tabular}
	\qquad
	\begin{tabular}[t]{l|l|c|c|c|c}
		$\pi_5$ & $w_5$ & $U_5$ & $D_5$ & $\ell_5$ & $k_5$ \\
		\hline
		$15342$ & $\varepsilon$ & $\{2\}$ & $\{4\}$ & $2$ & . \\
		$15243$ & $s_2$ & $\{3\}$ & $\{4\}$ & $3$ & . \\
		$15234$ & $s_2 \cdot s_3$ & $\{4\}$ & $\{3\}$ & $4$ & . \\
		$14235$ & $s_2 \cdot s_3 \cdot s_5$ & $\{5\}$ & $\{3\}$ & $3$ & \textcolor{red}{$\xcancel{\textcolor{black}{2}}$}
	\end{tabular}
	}
	\caption{$(\{2\},\{4 \})$-permutree sorting of~$\pi_4 \eqdef 54213$ and~$\pi_5 \eqdef 15342$.}
	\label{tab:algo2-2-4}
\end{table}
\end{example}

\begin{corollary}\label{coro:algorithmIntersection}
For any permutation~$\pi$ and any disjoint subsets~$U$ and~$D$ of~$\{2, \dots, n-1\}$, \cref{algo:permutreeSorting} returns a reduced word~$w$ accepted by $\automatonP(U,D)$ with the property that~$w$ is a reduced word for~$\pi$ if and only if~$\pi$ avoids~$jki$ for~$j \in U$ and~$kij$ for~$j \in D$.
\end{corollary}

\begin{proof}
This algorithm constructs a candidate reduced word for $\pi$ following the automaton~$\automatonP(U,D)$ and prioritizing healthy states over ill states. 
It begins by checking all possible transitions~$s_\ell$ that keep $\automatonP(U,D)$ in healthy states following \cref{lem:vincent4} (if condition in line~2).
Doing this in the intersection of automata translates to updating~$\ell$ to~$\ell+1$ when~$\ell \in U$ and $\ell+1$ to~$\ell$ when~$\ell+1 \in D$ (line~3).
When we have exhausted all these transitions, we need to go to an ill state of~$\automatonP(U,D)$, \ie to apply a transposition that sends at least one automaton of the intersection to an ill state.
If there is~$\ell+1 \in U$ (resp.~$\ell \in D$) such that~$s_\ell$ is a descent of~$\pi$ and~$\pi([\ell+1]) = [\ell+1]$ (resp.~$\pi([\ell-1]) = [\ell-1]$), then any reduced word for~$\pi$ is accepted by the automaton~$\automatonU(\ell)$ (resp.~$\automatonD(\ell)$) by \cref{prop:sameStateAcceptedReducedExpressionsRefined}\,(i). 
We can thus start with~$s_\ell$ and forget about the automaton~$\automatonU(\ell)$ (resp.~$\automatonD(\ell)$) (lines 4, 5 and 6).
Finally, if none of these options are possible, any reduced word for~$\pi$ will lead to a dead state in at least one of the automata, so that $\pi$ is not $(U,D)$-sortable.
We thus return the empty reduced word (line 7).
\end{proof}

\subsection{Generating trees}

As in \cref{subsec:trees}, we can define natural generating trees for the $(U,D)$-permutree minimal permutations.
Namely, fix an arbitrary priority order~$\prec$ on~$\{s_1, \dots, s_{n-1}\}$.
For an $(U,D)$-permutree minimal permutation~$\pi$, we denote by~$\pi(U, D, \prec)$ the $\prec$-lexicographic minimal reduced word for~$\pi$ that is accepted by~$\automatonP(U,D)$.
We denote by~$\lexmin(n, U, D, \prec)$ the set of reduced words of the form~$\pi(U, D, \prec)$ for all $(U,D)$-permutree minimal permutations~$\pi \in \fS_n$.
The same proof as that of \cref{prop:generatingTrees} shows that $\lexmin(n, U, D, \prec)$ is closed by prefix.
This yields a natural generating tree on~$\lexmin(n, U, D, \prec)$ where the parent of a reduced word~$w$ is obtained by deleting its last letter.
Replacing each reduced word by the corresponding permutation, this provides a generating tree for the $(U,D)$-permutree minimal permutations of~$\fS_n$.
\cref{fig:TreeCompleteOrientations} presents these generating trees for different values of~$U$ and~$D$.

\hvFloat[floatPos=p, capWidth=h, capPos=r, capAngle=90, objectAngle=90, capVPos=c, objectPos=c]{figure}
{
	\begin{tabular}{cccc}
		$\automatonU(2) \cap \automatonU(3)$
		&
		$\automatonU(2) \cap \automatonD(3)$
		&
		$\automatonU(3) \cap \automatonD(2)$
		&
		$\automatonD(2) \cap \automatonD(3)$
		\\[.5cm]
		\begin{tikzpicture}[shorten >=1pt, node distance=2cm, on grid, auto]
			\node[state,initial,accepting,minimum size=0.5cm] (hj)   {}; 
			\node[state,accepting,minimum size=0.5cm] (ij) [below= 1.5cm of hj] {}; 
			\node[state,minimum size=0.5cm] (dj) [below= 1.5cm of ij] {}; 
			\node[state,accepting,minimum size=0.5cm] (hj+1) [right= 1.5cm of hj] {};
			\node[state,accepting,minimum size=0.5cm] (ij+1) [below= 1.5cm of hj+1] {}; 
			\node[state,minimum size=0.5cm] (dj+1) [below= 1.5cm of ij+1] {}; 
			\node[state,accepting,minimum size=0.5cm] (hj+2) [right= 1.5cm of hj+1] {};
			\node[state,accepting,minimum size=0.5cm] (ij+2) [below= 1.5cm of hj+2] {}; 
			\path[->] 
				(hj) edge node [swap] {$s_1$} (ij)
					   edge node {$s_2$} (hj+1)
				(ij) edge node [swap] {$s_2$} (dj)
				(hj+1) edge node [swap] {$s_2$} (ij+1)
					 edge node {$s_3$} (hj+2)
				(ij+1) edge node [swap] {$s_3$} (dj+1)
				(hj+2) edge node [swap] {$s_3$} (ij+2);
		\end{tikzpicture}
		&
		\begin{tikzpicture}[shorten >=1pt, node distance=2cm, on grid, auto]
			\node[state,initial,accepting,minimum size=0.5cm] (hj)   {}; 
			\node[state,accepting,minimum size=0.5cm] (ij) [below= 1.5cm of hj] {}; 
			\node[state,minimum size=0.5cm] (dj) [below= 1.5cm of ij] {}; 
			\node[state,accepting,minimum size=0.5cm] (hj+1) [right= 1.5cm of hj] {};
			\node[state,accepting,minimum size=0.5cm] (ij+1) [below= 1.5cm of hj+1] {}; 
			\node[state,minimum size=0.5cm] (dj+1) [below= 1.5cm of ij+1] {}; 
			\node[state,accepting,minimum size=0.5cm] (hj+2) [right= 1.5cm of hj+1] {};
			\node[state,accepting,minimum size=0.5cm] (ij+2) [below= 1.5cm of hj+2] {}; 
			\path[->] 
				(hj) edge node [swap] {$s_1$} (ij)
					   edge node {$s_2$} (hj+1)
				(ij) edge node [swap] {$s_2$} (dj)
				(hj+1) edge node [swap] {$s_2$} (ij+1)
					 edge node {$s_3$} (hj+2)
				(ij+1) edge node [swap] {$s_3$} (dj+1)
				(hj+2) edge node [swap] {$s_3$} (ij+2);
		\end{tikzpicture}
		&
		\begin{tikzpicture}[shorten >=1pt, node distance=2cm, on grid, auto]
			\node[state,initial,accepting,minimum size=0.5cm] (hj)   {}; 
			\node[state,accepting,minimum size=0.5cm] (ij) [below= 1.5cm of hj] {}; 
			\node[state,minimum size=0.5cm] (dj) [below= 1.5cm of ij] {}; 
			\node[state,accepting,minimum size=0.5cm] (hj+1) [right= 1.5cm of hj] {};
			\node[state,accepting,minimum size=0.5cm] (ij+1) [below= 1.5cm of hj+1] {}; 
			\path[->] 
				(hj) edge node [swap] {$s_2$} (ij)
					   edge node {$s_3$} (hj+1)
				(ij) edge node [swap] {$s_3$} (dj)
				(hj+1) edge node [swap] {$s_3$} (ij+1);
		\end{tikzpicture}
		&
		\begin{tikzpicture}[shorten >=1pt, node distance=2cm, on grid, auto]
			\node[state,initial,accepting,minimum size=0.5cm] (hj)   {}; 
			\node[state,accepting,minimum size=0.5cm] (ij) [below= 1.5cm of hj] {}; 
			\node[state,minimum size=0.5cm] (dj) [below= 1.5cm of ij] {}; 
			\node[state,accepting,minimum size=0.5cm] (hj-1) [right= 1.5cm of hj] {};
			\node[state,accepting,minimum size=0.5cm] (ij-1) [below= 1.5cm of hj-1] {}; 
			\path[->] 
				(hj) edge node [swap] {$s_2$} (ij)
					   edge node {$s_1$} (hj-1)
				(ij) edge node [swap] {$s_1$} (dj)
				(hj-1) edge node [swap] {$s_1$} (ij-1);
		\end{tikzpicture}
		\\[.5cm]
		\begin{tikzpicture}[shorten >=1pt, node distance=2cm, on grid, auto]
			\node[state,initial,accepting,minimum size=0.5cm] (hj)   {}; 
			\node[state,accepting,minimum size=0.5cm] (ij) [below= 1.5cm of hj] {}; 
			\node[state,minimum size=0.5cm] (dj) [below= 1.5cm of ij] {}; 
			\node[state,accepting,minimum size=0.5cm] (hj+1) [right= 1.5cm of hj] {};
			\node[state,accepting,minimum size=0.5cm] (ij+1) [below= 1.5cm of hj+1] {}; 
			\path[->] 
				(hj) edge node [swap] {$s_2$} (ij)
					   edge node {$s_3$} (hj+1)
				(ij) edge node [swap] {$s_3$} (dj)
				(hj+1) edge node [swap] {$s_3$} (ij+1);
		\end{tikzpicture}
		&
		\begin{tikzpicture}[shorten >=1pt, node distance=2cm, on grid, auto]
			\node[state,initial,accepting,minimum size=0.5cm] (hj)   {}; 
			\node[state,accepting,minimum size=0.5cm] (ij) [below= 1.5cm of hj] {}; 
			\node[state,minimum size=0.5cm] (dj) [below= 1.5cm of ij] {}; 
			\node[state,accepting,minimum size=0.5cm] (hj-1) [right= 1.5cm of hj] {};
			\node[state,accepting,minimum size=0.5cm] (ij-1) [below= 1.5cm of hj-1] {}; 
			\node[state,minimum size=0.5cm] (dj-1) [below= 1.5cm of ij-1] {}; 
			\node[state,accepting,minimum size=0.5cm] (hj-2) [right= 1.5cm of hj-1] {};
			\node[state,accepting,minimum size=0.5cm] (ij-2) [below= 1.5cm of hj-2] {}; 
			\path[->] 
				(hj) edge node [swap] {$s_3$} (ij)
					   edge node {$s_2$} (hj-1)
				(ij) edge node [swap] {$s_2$} (dj)
				(hj-1) edge node [swap] {$s_2$} (ij-1)
					 edge node {$s_1$} (hj-2)
				(ij-1) edge node [swap] {$s_1$} (dj-1)
				(hj-2) edge node [swap] {$s_1$} (ij-2);
		\end{tikzpicture}
		&
		\begin{tikzpicture}[shorten >=1pt, node distance=2cm, on grid, auto]
			\node[state,initial,accepting,minimum size=0.5cm] (hj)   {}; 
			\node[state,accepting,minimum size=0.5cm] (ij) [below= 1.5cm of hj] {}; 
			\node[state,minimum size=0.5cm] (dj) [below= 1.5cm of ij] {}; 
			\node[state,accepting,minimum size=0.5cm] (hj-1) [right= 1.5cm of hj] {};
			\node[state,accepting,minimum size=0.5cm] (ij-1) [below= 1.5cm of hj-1] {}; 
			\path[->] 
				(hj) edge node [swap] {$s_2$} (ij)
					   edge node {$s_1$} (hj-1)
				(ij) edge node [swap] {$s_1$} (dj)
				(hj-1) edge node [swap] {$s_1$} (ij-1);
		\end{tikzpicture}
		&
		\begin{tikzpicture}[shorten >=1pt, node distance=2cm, on grid, auto]
			\node[state,initial,accepting,minimum size=0.5cm] (hj)   {}; 
			\node[state,accepting,minimum size=0.5cm] (ij) [below= 1.5cm of hj] {}; 
			\node[state,minimum size=0.5cm] (dj) [below= 1.5cm of ij] {}; 
			\node[state,accepting,minimum size=0.5cm] (hj-1) [right= 1.5cm of hj] {};
			\node[state,accepting,minimum size=0.5cm] (ij-1) [below= 1.5cm of hj-1] {}; 
			\node[state,minimum size=0.5cm] (dj-1) [below= 1.5cm of ij-1] {}; 
			\node[state,accepting,minimum size=0.5cm] (hj-2) [right= 1.5cm of hj-1] {};
			\node[state,accepting,minimum size=0.5cm] (ij-2) [below= 1.5cm of hj-2] {}; 
			\path[->] 
				(hj) edge node [swap] {$s_3$} (ij)
					   edge node {$s_2$} (hj-1)
				(ij) edge node [swap] {$s_2$} (dj)
				(hj-1) edge node [swap] {$s_2$} (ij-1)
					 edge node {$s_1$} (hj-2)
				(ij-1) edge node [swap] {$s_1$} (dj-1)
				(hj-2) edge node [swap] {$s_1$} (ij-2);
		\end{tikzpicture}
		\\[.5cm]
		\begin{tikzpicture}[xscale=.9, yscale=0.7, color=lightgray]
			\node[blue](P1234) at (0,0){1234};
			\node[blue](P2134) at (-1,1.5){2134};
			\node[blue](P1324) at ( 0,1.5){1324};
			\node[blue](P1243) at ( 1,1.5){1243};
			\node(P2314) at (-2,3){2314};
			\node[blue](P3124) at (-1,3){3124};
			\node[blue](P2143) at ( 0,3){2143};
			\node(P1342) at ( 1,3){1342};
			\node[blue](P1423) at ( 2,3){1423};
			\node[blue](P3214) at (-2.5,4.5){3214};
			\node(P2341) at (-1.5,4.5){2341};
			\node(P3142) at (-0.5,4.5){3142};
			\node(P2413) at ( 0.5,4.5){2413};
			\node[blue](P4123) at ( 1.5,4.5){4123};
			\node[blue](P1432) at ( 2.5,4.5){1432};
			\node(P3241) at (-2,6){3241};
			\node(P2431) at (-1,6){2431};
			\node(P3412) at ( 0,6){3412};
			\node[blue](P4213) at ( 1,6){4213};
			\node[blue](P4132) at ( 2,6){4132};
			\node(P3421) at (-1,7.5){3421};
			\node(P4231) at ( 0,7.5){4231};
			\node[blue](P4312) at ( 1,7.5){4312};
			\node[blue](P4321) at (0,9){4321};
			\draw[line width=0.5mm,blue](P1234) -- (P2134);
			\draw[line width=0.5mm,red](P1234) -- (P1324);
			\draw[line width=0.5mm,green](P1234) -- (P1243);
			\draw(P2134) -- (P2314);
			\draw[line width=0.5mm,green](P2134) -- (P2143);
			\draw[line width=0.5mm,blue](P1324) -- (P3124);
			\draw(P1324) -- (P1342);
			\draw(P1243) -- (P2143);
			\draw[line width=0.5mm,red](P1243) -- (P1423);
			\draw(P2314) -- (P3214);
			\draw(P2314) -- (P2341);
			\draw[line width=0.5mm,red](P3124) -- (P3214);
			\draw(P3124) -- (P3142);
			\draw(P2143) -- (P2413);
			\draw(P1342) -- (P3142);
			\draw(P1342) -- (P1432);
			\draw[line width=0.5mm,blue](P1423) -- (P4123);
			\draw[line width=0.5mm,green](P1423) -- (P1432);
			\draw(P3214) -- (P3241);
			\draw(P2341) -- (P3241);
			\draw(P2341) -- (P2431);
			\draw(P3142) -- (P3412);
			\draw(P2413) -- (P4213);
			\draw(P2413) -- (P2431);
			\draw[line width=0.5mm,red](P4123) -- (P4213);
			\draw[line width=0.5mm,green](P4123) -- (P4132);
			\draw(P1432) -- (P4132);
			\draw(P3241) -- (P3421);
			\draw(P2431) -- (P4231);
			\draw(P3412) -- (P4312);
			\draw(P3412) -- (P3421);
			\draw(P4213) -- (P4231);
			\draw[line width=0.5mm,red](P4132) -- (P4312);
			\draw(P3421) -- (P4321);
			\draw(P4231) -- (P4321);
			\draw[line width=0.5mm,green](P4312) -- (P4321);
		\end{tikzpicture}
		&
		\begin{tikzpicture}[xscale=.9, yscale=0.7, color=lightgray]
			\node[blue](P1234) at (0,0){1234};
			\node[blue](P2134) at (-1,1.5){2134};
			\node[blue](P1324) at ( 0,1.5){1324};
			\node[blue](P1243) at ( 1,1.5){1243};
			\node(P2314) at (-2,3){2314};
			\node[blue](P3124) at (-1,3){3124};
			\node[blue](P2143) at ( 0,3){2143};
			\node[blue](P1342) at ( 1,3){1342};
			\node(P1423) at ( 2,3){1423};
			\node[blue](P3214) at (-2.5,4.5){3214};
			\node(P2341) at (-1.5,4.5){2341};
			\node[blue](P3142) at (-0.5,4.5){3142};
			\node(P2413) at ( 0.5,4.5){2413};
			\node(P4123) at ( 1.5,4.5){4123};
			\node[blue](P1432) at ( 2.5,4.5){1432};
			\node(P3241) at (-2,6){3241};
			\node(P2431) at (-1,6){2431};
			\node[blue](P3412) at ( 0,6){3412};
			\node(P4213) at ( 1,6){4213};
			\node(P4132) at ( 2,6){4132};
			\node[blue](P3421) at (-1,7.5){3421};
			\node(P4231) at ( 0,7.5){4231};
			\node[blue](P4312) at ( 1,7.5){4312};
			\node[blue](P4321) at (0,9){4321};
			\draw[line width=0.5mm,blue](P1234) -- (P2134);
			\draw[line width=0.5mm,red](P1234) -- (P1324);
			\draw[line width=0.5mm,green](P1234) -- (P1243);
			\draw(P2134) -- (P2314);
			\draw[line width=0.5mm,green](P2134) -- (P2143);
			\draw[line width=0.5mm,blue](P1324) -- (P3124);
			\draw[line width=0.5mm,green](P1324) -- (P1342);
			\draw(P1243) -- (P2143);
			\draw(P1243) -- (P1423);
			\draw(P2314) -- (P3214);
			\draw(P2314) -- (P2341);
			\draw[line width=0.5mm,red](P3124) -- (P3214);
			\draw[line width=0.5mm,green](P3124) -- (P3142);
			\draw(P2143) -- (P2413);
			\draw(P1342) -- (P3142);
			\draw[line width=0.5mm,red](P1342) -- (P1432);
			\draw(P1423) -- (P4123);
			\draw(P1423) -- (P1432);
			\draw(P3214) -- (P3241);
			\draw(P2341) -- (P3241);
			\draw(P2341) -- (P2431);
			\draw[line width=0.5mm,red](P3142) -- (P3412);
			\draw(P2413) -- (P4213);
			\draw(P2413) -- (P2431);
			\draw(P4123) -- (P4213);
			\draw(P4123) -- (P4132);
			\draw(P1432) -- (P4132);
			\draw(P3241) -- (P3421);
			\draw(P2431) -- (P4231);
			\draw[line width=0.5mm,blue](P3412) -- (P4312);
			\draw[line width=0.5mm,green](P3412) -- (P3421);
			\draw(P4213) -- (P4231);
			\draw(P4132) -- (P4312);
			\draw(P3421) -- (P4321);
			\draw(P4231) -- (P4321);
			\draw[line width=0.5mm,green](P4312) -- (P4321);
		\end{tikzpicture}
		&
		\begin{tikzpicture}[xscale=.9, yscale=0.7, color=lightgray]
			\node[blue](P1234) at (0,0){1234};
			\node[blue](P2134) at (-1,1.5){2134};
			\node[blue](P1324) at ( 0,1.5){1324};
			\node[blue](P1243) at ( 1,1.5){1243};
			\node[blue](P2314) at (-2,3){2314};
			\node(P3124) at (-1,3){3124};
			\node[blue](P2143) at ( 0,3){2143};
			\node(P1342) at ( 1,3){1342};
			\node[blue](P1423) at ( 2,3){1423};
			\node[blue](P3214) at (-2.5,4.5){3214};
			\node(P2341) at (-1.5,4.5){2341};
			\node(P3142) at (-0.5,4.5){3142};
			\node[blue](P2413) at ( 0.5,4.5){2413};
			\node(P4123) at ( 1.5,4.5){4123};
			\node[blue](P1432) at ( 2.5,4.5){1432};
			\node(P3241) at (-2,6){3241};
			\node[blue](P2431) at (-1,6){2431};
			\node(P3412) at ( 0,6){3412};
			\node[blue](P4213) at ( 1,6){4213};
			\node(P4132) at ( 2,6){4132};
			\node(P3421) at (-1,7.5){3421};
			\node[blue](P4231) at ( 0,7.5){4231};
			\node(P4312) at ( 1,7.5){4312};
			\node[blue](P4321) at (0,9){4321};
			\draw[line width=0.5mm,blue](P1234) -- (P2134);
			\draw[line width=0.5mm,red](P1234) -- (P1324);
			\draw[line width=0.5mm,green](P1234) -- (P1243);
			\draw[line width=0.5mm,red](P2134) -- (P2314);
			\draw[line width=0.5mm,green](P2134) -- (P2143);
			\draw(P1324) -- (P3124);
			\draw(P1324) -- (P1342);
			\draw(P1243) -- (P2143);
			\draw[line width=0.5mm,red](P1243) -- (P1423);
			\draw[line width=0.5mm,blue](P2314) -- (P3214);
			\draw(P2314) -- (P2341);
			\draw(P3124) -- (P3214);
			\draw(P3124) -- (P3142);
			\draw[line width=0.5mm,red](P2143) -- (P2413);
			\draw(P1342) -- (P3142);
			\draw(P1342) -- (P1432);
			\draw(P1423) -- (P4123);
			\draw[line width=0.5mm,green](P1423) -- (P1432);
			\draw(P3214) -- (P3241);
			\draw(P2341) -- (P3241);
			\draw(P2341) -- (P2431);
			\draw(P3142) -- (P3412);
			\draw[line width=0.5mm,blue](P2413) -- (P4213);
			\draw[line width=0.5mm,green](P2413) -- (P2431);
			\draw(P4123) -- (P4213);
			\draw(P4123) -- (P4132);
			\draw(P1432) -- (P4132);
			\draw(P3241) -- (P3421);
			\draw(P2431) -- (P4231);
			\draw(P3412) -- (P4312);
			\draw(P3412) -- (P3421);
			\draw[line width=0.5mm,green](P4213) -- (P4231);
			\draw(P4132) -- (P4312);
			\draw(P3421) -- (P4321);
			\draw[line width=0.5mm,red](P4231) -- (P4321);
			\draw(P4312) -- (P4321);
		\end{tikzpicture}
		&
    	\begin{tikzpicture}[xscale=.9, yscale=0.7, color=lightgray]
			\node[blue](P1234) at (0,0){1234};
			\node[blue](P2134) at (-1,1.5){2134};
			\node[blue](P1324) at ( 0,1.5){1324};
			\node[blue](P1243) at ( 1,1.5){1243};
			\node[blue](P2314) at (-2,3){2314};
			\node(P3124) at (-1,3){3124};
			\node[blue](P2143) at ( 0,3){2143};
			\node[blue](P1342) at ( 1,3){1342};
			\node(P1423) at ( 2,3){1423};
			\node[blue](P3214) at (-2.5,4.5){3214};
			\node[blue](P2341) at (-1.5,4.5){2341};
			\node(P3142) at (-0.5,4.5){3142};
			\node(P2413) at ( 0.5,4.5){2413};
			\node(P4123) at ( 1.5,4.5){4123};
			\node[blue](P1432) at ( 2.5,4.5){1432};
			\node[blue](P3241) at (-2,6){3241};
			\node[blue](P2431) at (-1,6){2431};
			\node(P3412) at ( 0,6){3412};
			\node(P4213) at ( 1,6){4213};
			\node(P4132) at ( 2,6){4132};
			\node[blue](P3421) at (-1,7.5){3421};
			\node(P4231) at ( 0,7.5){4231};
			\node(P4312) at ( 1,7.5){4312};
			\node[blue](P4321) at (0,9){4321};
			\draw[line width=0.5mm,blue](P1234) -- (P2134);
			\draw[line width=0.5mm,red](P1234) -- (P1324);
			\draw[line width=0.5mm,green](P1234) -- (P1243);
			\draw[line width=0.5mm,red](P2134) -- (P2314);
			\draw[line width=0.5mm,green](P2134) -- (P2143);
			\draw(P1324) -- (P3124);
			\draw[line width=0.5mm,green](P1324) -- (P1342);
			\draw(P1243) -- (P2143);
			\draw(P1243) -- (P1423);
			\draw[line width=0.5mm,blue](P2314) -- (P3214);
			\draw[line width=0.5mm,green](P2314) -- (P2341);
			\draw(P3124) -- (P3214);
			\draw(P3124) -- (P3142);
			\draw(P2143) -- (P2413);
			\draw(P1342) -- (P3142);
			\draw[line width=0.5mm,red](P1342) -- (P1432);
			\draw(P1423) -- (P4123);
			\draw(P1423) -- (P1432);
			\draw[line width=0.5mm,green](P3214) -- (P3241);
			\draw(P2341) -- (P3241);
			\draw[line width=0.5mm,red](P2341) -- (P2431);
			\draw(P3142) -- (P3412);
			\draw(P2413) -- (P4213);
			\draw(P2413) -- (P2431);
			\draw(P4123) -- (P4213);
			\draw(P4123) -- (P4132);
			\draw(P1432) -- (P4132);
			\draw[line width=0.5mm,red](P3241) -- (P3421);
			\draw(P2431) -- (P4231);
			\draw(P3412) -- (P4312);
			\draw(P3412) -- (P3421);
			\draw(P4213) -- (P4231);
			\draw(P4132) -- (P4312);
			\draw[line width=0.5mm,blue](P3421) -- (P4321);
			\draw(P4231) -- (P4321);
			\draw(P4312) -- (P4321);
    	\end{tikzpicture}
	\end{tabular}
}
{Generating trees for the $(U,D)$-permutree minimal permutations of~$\fS_4$, with priority order $s_1 \prec s_2 \prec s_3$.}
{fig:TreeCompleteOrientations}


\section{Permutree sorting versus Coxeter sorting}\label{sec:coxeterSortable}

In this section, we discuss the particular case when~$U$ and~$D$ form a partition of~$\{2, \dots, n-1\}$.
In that situation, we connect the $(U,D)$-permutree sorting with the $c$-sorting of N.~Reading~\cite{Reading-sortableElements}.

\subsection{Coxeter sorting word and Coxeter sortable permutations}\label{subsec:csorting}

We first recall the theory of $c$-sorting developed by N.~Reading in~\cite{Reading-sortableElements}.
While it was defined in arbitrary finite Coxeter groups, we focus on the symmetric group in this presentation.

We consider a \defn{Coxeter element}~$c$ of~$\fS_n$, \ie the product of all simple transpositions $\{s_1, \dots, s_{n-1}\}$ in an arbitrary order.
For a permutation $\pi \in \fS_n$, the \defn{$c$-sorting word} $\pi(c)$ is the lexicographically smallest reduced word for $\pi$ in the infinite word $c^\infty = c \cdot c \cdot c \cdot c \cdots$.
Note that strictly speaking, $\pi(c)$ depends on a reduced word for~$c$, not only on the Coxeter element~$c$.
Here, we assume that we have chosen a reduced word and hide this dependence.
We let~$I_1, \dots, I_p$ denote the subsets of~$[n-1]$ such that~$\pi(c) = c_{I_1} \cdot c_{I_2} \cdots c_{I_p}$ where~$c_I$ is the subword of~$c$ obtained by keeping only the letters~$s_i$ for~$i \in I$.
The permutation $\pi$ is \defn{$c$-sortable} if~$I_1 \supseteq I_2 \supseteq \dots \supseteq I_p$.
Note that this does not depend on the choice of the reduced word~$c$, only on the Coxeter element~$c$.

For our proofs we will need some simple yet useful facts from \cite{Reading-sortableElements} on how prefixes of words influence sortability.

\begin{lemma}\label{lem:coxeterElementFacts}
Consider a Coxeter element of the form~$c = s_\ell \cdot d$ and let $\pi \in \fS_n$. Then
\begin{itemize}
	\item if $\pi = s_\ell \cdot \tau$ with $\ell(\pi) = \ell(\tau) + 1$, then $\pi(c) = s_\ell \cdot \tau(d \cdot s_\ell)$,
	\item otherwise, $\pi(c) = \pi(d \cdot s_\ell)$.
\end{itemize}
\end{lemma}

\begin{lemma}\label{lem:vincent8}
Let~$s_i \ne s_j$ be two letters that appear in the $c$-sorting word $\pi(c)$ of a $c$-sortable permutation~$\pi$. Then
\begin{enumerate} 
	\item if $s_i$ appears before $s_j$ in $c$, then $s_i$ appears before $s_j$ in $\pi(c)$,
	\item if $s_j$ does not appear in between two occurrences of $s_i$ in $\pi(c)$, then it does not appear after these occurrences either.
\end{enumerate}
\end{lemma}

\begin{proof}
We deal with the two statements separately:
	\begin{enumerate} 
		\item Immediate from the definition since both $s_i$ and $s_j$ appear in $\pi(c)$.
		\item Since $\pi$ is $c$-sortable, $\pi(c)$ is formed by a succession of subwords that are nested. Thus if $s_j$ were to appear after the occurrences of $s_i$ it would have to appear between them as well.
		\qedhere
	\end{enumerate}
\end{proof}

\subsection{Coxeter sorting via permutree automata}

A Coxeter element~$c$ of~$\fS_n$ defines a partition ${\{2, \dots, n-1\} = U_c \sqcup D_c}$, where~$U_c$ (resp.~$D_c$) consists of the elements~$j \in \{2, \dots, n-1\}$ such that~$s_j$ appears before (resp.~after) $s_{j-1}$ in~$c$.
For instance, when~$c = s_2 \cdot s_5 \cdot s_4 \cdot s_3 \cdot s_1 \cdot s_6$, we obtain~$U_c = \{2,4,5\}$ and~$D_c = \{3,6\}$.
Said differently, $j \in U$ (resp.~$j \in D$) if~$c$ is accepted by~$\automatonU(j)$ but not by~$\automatonD(j)$ (resp.~by~$\automatonD(j)$ but not by~$\automatonU(j)$).
The goal of this section is the following connection between the $c$-sorting of \cref{subsec:csorting} and the $(U_c, D_c)$-permutree sorting of \cref{sec:intersectionsAutomata}.

\begin{theorem}
\label{thm:csorting}
For any Coxeter element~$c$ and any permutation~$\pi$, the following assertions are equivalent:
\begin{enumerate}[(i)]
	\item $\pi$ is $c$-sortable,
	\item the $c$-sorting word~$\pi(c)$ is accepted by the automaton~$\automatonP(U_c, D_c)$,
	\item there exists a reduced word for~$\pi$ accepted by the automaton~$\automatonP(U_c, D_c)$,
	\item for each~$j \in \{2, \dots, n-1\}$, there exists a reduced word for $\pi$ that is accepted by the automaton~$\automatonU(j)$ if~$j \in U_c$ and~$\automatonD(j)$ if~$j \in D_c$,
	\item $\pi$ avoids $jki$ for~$j \in U_c$ and $kij$ for~$j \in D_c$.
\end{enumerate}
\end{theorem}

The equivalences (iii) $\iff$ (iv) $\iff$ (v) were already established earlier.
Here, we aim at identifying the $c$-sorting word as a reduced word for~$\pi$ accepted by~$\automatonP(U_c, D_c)$.
We split the proof of the equivalences (i) $\iff$ (ii) $\iff$ (iii) into the following few lemmas.

\begin{lemma}\label{lem:vincent9}
The $c$-sorting word of a $c$-sortable permutation is recognized by~$\automatonU(j)$ for~$j \in U_c$ and by~$\automatonD(j)$ for~$j \in D_c$.
\end{lemma}

\begin{proof}
Consider~$j \in U_c$ (the proof for~$j \in D_c$ is symmetric).
We distinguish two possible cases:
\begin{itemize}
	\item If $\pi(c)$ contains no $s_{j-1}$, then $\pi(c)$ either remains in the first healthy state or ends in the first ill state of~$\automatonU(j)$.
	\item If $\pi(c)$ contains $s_{j-1}$, then by \cref{lem:vincent8}\,(1) $s_{j-1}$ appears before $s_j$ in $\pi(c)$ and $\pi(c)$ leads to the second healthy state of~$\automatonU(j)$. From here on out, notice that $\pi(c)$ cannot end at a dead state because of \cref{lem:vincent8}\,(2). 
\end{itemize}
In both cases, $\pi(c)$ is accepted by~$\automatonU(j)$.
\end{proof}

\begin{lemma}\label{lem:Vincent10}
A permutation $\pi \in \fS_n$ whose $c$-sorting word $\pi(c)$ is accepted by~$\automatonP(U_c, D_c)$ is $c$-sortable.
\end{lemma}

\begin{proof}
Suppose that $\pi$ is not $c$-sortable. We will find an automaton that rejects $\pi(c)$ among the automata~$\automatonU(j)$ for~$j \in U_c$ and~$\automatonD(j)$ for~$j \in D_c$.
Once again, we work by induction on the length of $\pi$ and the size of~$c$.
Let~$s_\ell$ be the first letter of~$c$ and write~$c = s_\ell \cdot d$.
Since~$s_\ell$ appears before~$s_{\ell-1}$ and~$s_{\ell+1}$ in~$c$, we have~$\ell \in U_c$ and~$\ell+1 \in D_c$.
Moreover, the letter~$s_\ell$ yields to the next healthy state in both automata~$\automatonU(\ell)$ and~$\automatonD(\ell+1)$, and remains in the initial state for all other automata~$\automatonU(j)$ for~$j \in U_c \ssm \{\ell\}$ and~$\automatonD(j)$ for~$j \in D_c \ssm \{\ell+1\}$.
We now distinguish two cases, depending on whether~$\ell$ and~$\ell+1$ are reversed in~$\pi$.

Assume first that $\ell$ and $\ell+1$ are reversed in $\pi$ and write $\pi = s_\ell \cdot \tau$.
We then have $\pi(c) = s_\ell \cdot \tau(d \cdot s_\ell)$ by \cref{lem:coxeterElementFacts}, so that~$\tau$ is not $d \cdot s_\ell$-sortable.
By induction hypothesis, $\tau(d \cdot s_\ell)$ is rejected by one of the automata~$\automatonU(j)$ for~$j \in U_{d \cdot s_\ell}$ and~$\automatonD(j)$ for~$j \in D_{d \cdot s_\ell}$.
Since~$U_{d \cdot s_\ell} = U_c \symdif \{\ell, \ell+1\}$ and~$D_{d \cdot s_\ell} = D_c \symdif \{\ell, \ell+1\}$, \cref{lem:vincent2,lem:vincent4} ensure that $\pi(c) = s_\ell \cdot \tau(d \cdot s_\ell)$ is rejected by one of the automata~$\automatonU(j)$ for~$j \in U_c$ and~$\automatonD(j)$ for~$j \in D_c$.

Assume now that $\ell$ and $\ell+1$ are not reversed in $\pi$.
Then~$\pi(c)$ does not use~$s_\ell$ and~$\pi$ is not $d$-sortable in~$W_{\langle s_\ell \rangle}$.
By induction hypothesis, $\pi(c)$ is rejected by one of the automata~$\automatonU(j)$ for~$j \in U_d$ and~$\automatonD(j)$ for~$j \in D_d$.
This concludes the proof since~$U_d \subseteq U_c$ and~$D_d \subseteq D_c$.
\end{proof}

\begin{lemma}\label{lem:Vincent11}
If a permutation~$\pi \in \fS_n$ admits a reduced word accepted by~$\automatonP(U_c, D_c)$, then its $c$-sorting word~$\pi(c)$ is accepted by~$\automatonP(U_c, D_c)$.
\end{lemma}

\begin{proof}
Once agin, we work by induction on the length of $\pi$.
Let~$s_\ell$ be the first letter of~$c$ and write~$c = s_\ell \cdot d$.
Since~$s_\ell$ appears before~$s_{\ell-1}$ and~$s_{\ell+1}$ in~$c$, we have~$\ell \in U_c$ and~$\ell+1 \in D_c$.
Moreover, the letter~$s_\ell$ yields to the next healthy state in both automata~$\automatonU(\ell)$ and~$\automatonD(\ell+1)$, and remains in the initial state for all other automata~$\automatonU(j)$ for~$j \in U_c \ssm \{\ell\}$ and~$\automatonD(j)$ for~$j \in D_c \ssm \{\ell+1\}$.
We now distinguish two cases, depending on whether~$\ell$ and~$\ell+1$ are reversed in~$\pi$.

Assume first that $\ell$ and $\ell+1$ are reversed in $\pi$ and write $\pi = s_\ell \cdot \tau$. Using \cref{lem:coxeterElementFacts} it suffices now to show that after $s_\ell$, there is a reduced word for $\tau$ accepted by the automata. For each $j$, since the automaton $\automatonD(j)$ or $\automatonU(j)$ that we see accepts at least one reduced word for $\pi$ and $s_\ell$ does not lead to a ill state, it also accepts a reduced word for $\pi$ starting with $s_\ell$ by \cref{prop:algorithm}. Observe moreover that:
\begin{itemize}
\item $\tau$ admits a reduced word accepted by~$\automatonU(j)$ for each~$j \notin \{\ell, \ell+1\}$ by \cref{lem:vincent2}. This lines up with the fact that the order of $s_{j-1}$ and $s_j$ has not changed from $c = s_\ell \cdot d$ to~$d \cdot s_\ell$.
\item $\tau$ admits a reduced word accepted by~$\automatonU(\ell+1)$ and a reduced word accepted by~$\automatonD(\ell)$ by \cref{lem:vincent4}. This fits the fact that $\ell$ now appears after $\ell-1$ and $\ell+1$ in $d \cdot s_\ell$.
\end{itemize}
By induction, we obtain that $\tau(d \cdot s_\ell)$ is accepted by~$\automatonP(U_{d \cdot s_\ell}, D_{d \cdot s_\ell})$, so that $\pi(c) = s_\ell \cdot \tau(d \cdot s_\ell)$ is accepted~$\automatonP(U_c, D_c)$.
	
Assume now that $\ell$ and $\ell+1$ are not reversed in $\pi$. We want to show that $s_\ell$ never appears in the reduced words for $\pi$, \ie that $\pi([\ell]) = [\ell]$ and $\pi([n] \ssm [\ell]) = [n] \ssm [\ell]$. Otherwise, the reduced word $w_\ell$ accepted by $\automatonU(\ell)$ would see first a $s_{\ell+1}$, and then a $s_\ell$ before it sees any $s_{\ell-1}$, so that we would have an inversion $k \ell$ in $\pi$ for some $\ell < k$. Similarly, the reduced word $w_{\ell+1}$ accepted by $\automatonD(\ell+1)$ should see first a $s_{\ell-1}$ and then a $s_\ell$ before it sees any $s_{\ell+1}$, so that we would have an inversion $(\ell+1) i$ in $\pi$ for some $i < \ell+1$. Since $\ell$ and $\ell+1$ are not reversed, we see $k \ell (\ell+1) i$ which contradicts twice \cref{thm:patternAvoidance}. We conclude that this case never happens, so that we can work in the parabolic subgroup of permutations that never use $s_\ell$ in their reduced words.
\end{proof}

\subsection{Some negative observations}

We conclude this paper with some negative observations and warnings about the connection between $c$-sorting and $(U_c,D_c)$-permutree sorting.
First, we want to underline that using $c$-sorting words to test whether a permutation avoids~$jki$ or $kij$ for a fixed~$j$ is dangerous for the following two reasons.

\begin{remark}
Even if a permutation~$\pi$ avoids~$jki$ (resp.~$kij$) for a given~$j$, there might be no Coxeter element~$c$ for which $\pi$ is $c$-sortable and~$j \in U_c$ (resp.~$j \in D_c$).
For instance, the permutation $41325 \in \fS_5$ avoids $2ki$ and~$ki4$, but contains~$352$ and~$413$, so it is not $c$-sortable for any Coxeter element~$c$.
\end{remark}

\begin{remark}
When a permutation~$\pi$ is not $c$-sortable, there might exist~$j \in U_c$ (resp.~$j \in D_c$) for which the $c$-sorting word~$\pi(c)$ is not accepted by~$\automatonU(j)$ (resp.~$\automatonD(j)$) even if~$\pi$ avoids~$jki$ (resp.~$kij$).
For instance, consider~$c = s_2 \cdot s_1 \cdot s_3$ and~$\pi = 4213 = s_3 \cdot s_1 \cdot s_2 \cdot s_1 =  s_3 \cdot s_2 \cdot s_1 \cdot s_2 =  s_1 \cdot s_3 \cdot s_2 \cdot s_1$.
Then~$2 \in U_c$, and the $c$-sorting word~$\pi(c) = s_1 \cdot s_3 \cdot s_2 \cdot s_1$ is rejected by~$\automatonU(2)$ while~$\pi$ contains no~$2ki$ (and indeed~$s_3 \cdot s_2 \cdot s_1 \cdot s_2$ is accepted by~$\automatonU(2)$).
\end{remark}

We conclude the paper with an observation about sorting networks and permutree sorting.

\begin{remark}
Given a Coxeter element $c$, the word $c^\infty$ which is used to compute $\pi(c)$ is a \defn{sorting network}. This means that we decide \emph{beforehand} a list of transpositions to apply if appropriate. On the other hand, the permutree sorting given in \cref{algo:permutreeSorting} is \emph{not} a sorting network. Indeed, the order on transpositions depends on the permutation and more specifically on the state of the automaton we are at. A natural question then occurs: can we replace the permutree sorting algorithm by a sorting network? Or said differently, when~$U$ and~$D$ are disjoint but do not cover~$\{2, \dots, n-1\}$, can we find a word $\tilde{c}$ which plays the role of $c^\infty$ in the sense that looking at $\pi(\tilde{c})$ would be enough to check whether $\pi$ is accepted by $\automatonP(U,D)$?

The answer is negative in general. A counter-example is found for $n = 5$, $U = \{ 2\}$, and $D=\{ 4 \}$. In this case one can check through computer exploration that no reduced word $\tilde{c}$ of the maximal permutation $54321$ can be used as a sorting network. Namely, for all choices of $\tilde{c}$, there exist a permutation $\pi$ which is accepted by $\automatonP(U,D)$ whereas the reduced word $\pi(\tilde{c})$ is rejected. The healthy states of $\automatonP(\{ 2\},\{ 4 \})$ are shown in \cref{fig:automataProduct}. We see that accepted reduced words can start with either $s_2$ or $s_3$. For some permutations, such as $54213$ shown in \cref{exm:algo2}, all accepted reduced words start with $s_3$ whereas for some other permutations such as $35421$, all accepted reduced words start with $s_2$. This eventually leads to an empty intersection for the choice of $\tilde{c}$.

Nevertheless, it seems interesting to study in which case the answer is positive. The Cambrien case with the $c$-sorting word when $U$ and~$D$ form a partition of~$\{2, \dots, n-1\}$ is one example. The case where $|U| + |D| = 1$ is another one. This is the case corresponding to \cref{thm:patternAvoidance} where we have only one automaton. In this case, we can construct a word $\tilde{c}$ by reading the healthy states of the automaton linearly, adding at each state the word $(s_{i_1} \cdots s_{i_k})^k s_j$ where $s_{i_1}, \dots, s_{i_k}$ are the looping transitions and $s_j$ is the transition going to the next healthy state. This process gives a prefix that can be extended in any way to obtain a proper sorting word $\tilde{c}$. For example, if $U = \{2\}$, we obtain the prefix $s_3 \cdot s_2 \cdot s_1 \cdot s_3$ and indeed $s_3 \cdot s_2 \cdot s_1 \cdot s_3 \cdot s_2 \cdot s_1$ acts as a sorting network equivalent to the $(\{2\},\varnothing)$-permutree sorting. This process actually seems to extend to all cases where, at each healthy state of the intersection automaton, the choices for the healthy transitions commute. For example, in the case where $n=5$, $U = \{ 4 \}$ and $D = \{ 2 \}$ as illustrated in \cref{fig:automataProduct}, the word $s_1 \cdot s_2 \cdot s_4 \cdot s_3 \cdot s_2 \cdot s_1 \cdot s_4 \cdot s_3 \cdot s_2$ gives a proper sorting network. 
\end{remark}


\bibliographystyle{alpha}
\bibliography{permutreeSortablePermutations}
\newpage

\end{document}